\documentclass[a4paper,UKenglish,cleveref,autoref,thm-restate]{lipics-v2021}

\hideLIPIcs  

\usepackage{tikz}
\usetikzlibrary{positioning, fit, calc, shapes.geometric}
\newcommand{\NP}{{\sf NP}}
\newcommand{\sP}{{\sf P}}
\newcommand{\ssi}{\subseteq_i}

\theoremstyle{plain}
\newtheorem{open}[theorem]{Open Problem}

\title{Optimal b-Colourings and Fall Colourings in  \texorpdfstring{$H$}{H}-Free Graphs\footnote{An extended abstract of this paper will appear in the Proceedings of WG 2026~\cite{AELMMP26WG}.}}

\titlerunning{Optimal b-Colourings and Fall Colourings in \texorpdfstring{$H$}{H}-Free Graphs}

\author{Jungho Ahn}{Department of Computer Engineering, Inha University, Incheon, South Korea}{junghoahn@inha.ac.kr}{https://orcid.org/0000-0003-0511-1976}{Supported by Leverhulme Trust Research Project Grant RPG-2024-182 and INHA University Research Grant.}

\author{Tala Eagling-Vose}{Department of Computer Science, Durham University, Durham, UK}{tala.j.eagling-vose@durham.ac.uk}{https://orcid.org/0009-0008-0346-7032}{}

\author{Felicia Lucke}{ENS Lyon, Lyon, France}{felicia.lucke@ens-lyon.fr}{https://orcid.org/0000-0002-9860-2928}{Supported by Postdoc.Mobility Grant 230578.}

\author{David Manlove}{School of Computing Science, University of Glasgow, Glasgow, UK}{david.manlove@glasgow.ac.uk}{https://orcid.org/0000-0001-6754-7308}{}

\author{Fabricio Mendoza Granada}{School of Computing Science, University of Glasgow, Glasgow, UK}{f.mendoza-granada.1@research.gla.ac.uk}{https://orcid.org/0009-0007-9123-1726}{Supported by a scholarship from the College of Science and Engineering, University of Glasgow, and Becas Carlos Antonio Lopez (BECAL), Paraguay.}

\author{Dani\"el Paulusma}{Department of Computer Science, Durham University, Durham, UK}{daniel.paulusma@durham.ac.uk}{https://orcid.org/0000-0001-5945-9287}{Supported by Leverhulme Trust Research Project Grant RPG-2024-182.}

\authorrunning{J. Ahn, T. Eagling-Vose, F. Lucke, D. Manlove, F. Mendoza, D. Paulusma} 

\Copyright{Jungho Ahn, Tala Eagling-Vose, Felicia Lucke, David Manlove, Fabricio Mendoza, Dani\"el Paulusma}



\acknowledgements{The authors would like to thank the reviewers for their valuable comments.}

\nolinenumbers

\begin{document}

\maketitle

\begin{abstract}
In a colouring of a graph, a vertex is 
\emph{b-chromatic} if it is adjacent to a vertex of every other colour.  We consider
four well-studied colouring problems:
{\sc b-Chromatic Number}, {\sc Tight b-Chromatic Number},
{\sc Fall Chromatic Number} and {\sc Fall Achromatic Number},
which fit into a framework based on  whether every colour class has (i) at least one b-chromatic vertex, (ii) exactly one b-chromatic vertex, or (iii) all of its vertices being b-chromatic.
By combining known and new results, we fully classify the computational complexity of {\sc b-Chromatic Number}, 
{\sc Fall Chromatic Number} and {\sc Fall Achromatic Number}
in $H$-free graphs.  
For {\sc Tight b-Chromatic Number} in $H$-free graphs, we develop a general technique to determine
new graphs~$H$, for which the problem is polynomial-time solvable, and we
also determine new graphs $H$, for which the problem is still \NP-complete.  
We show, for the first time, the existence of a graph~$H$ such that
in $H$-free graphs, {\sc b-Chromatic Number} is \NP-hard, while {\sc Tight b-Chromatic Number} is polynomial-time solvable.
\end{abstract}

\section{Introduction}\label{s-intro}
Colouring is an extensively studied area of graph theory, mostly focussing on minimising the number of colours used~\cite{JT95,Lewis21}.
Namely, a \emph{colouring} in a graph $G$ is a mapping $c:V\to \mathbb{Z}^+$ such that $c(u)\neq c(v)$ for every edge $uv\in E$. If $|c(V)|=k$, then $c$ is a {\it $k$-colouring}, and the {\it chromatic number} is the smallest $k$ such that $G$ has a $k$-colouring.
However, there have also been well-studied graph colouring concepts that involve maximising the number of colours $k$ subject to additional conditions, such as the achromatic number~\cite{kortsarz2007improved}, {b-chromatic number}~\cite{IM99,jakovac2018b}, Grundy number~\cite{bonnet2018complexity} and {fall achromatic number}~\cite{dunbar2000fall}.
In particular, the \emph{b-chromatic number} of a graph $G$ is the maximum integer $k$ for which $G$ admits a $k$-colouring where every colour $i$ has a vertex $v_i$ of colour $i$ that is adjacent to a vertex of every colour $j\neq i$.
Such a vertex~$v_i$ is called
\emph{b-chromatic},
and a colouring where every colour class has a b-chromatic vertex is called a \emph{b-colouring}.
See Figure~\ref{f-examples} for an example of a b-colouring.

\begin{figure}[t]
    \centering
    \begin{minipage}{0.4\textwidth}
        \centering
        \includegraphics[width=.75\textwidth]{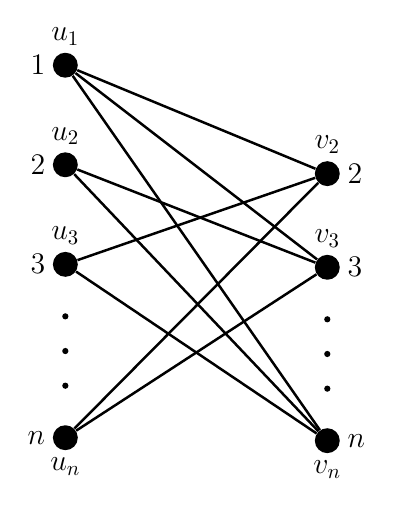}
    \end{minipage}
    \hfill
    \begin{minipage}{0.4\textwidth}
        \centering
        \includegraphics[width=.75\textwidth]{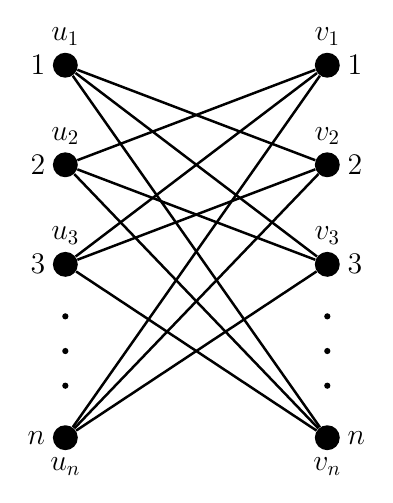}
    \end{minipage}
    \caption{An example of a b-colouring of a $(2n-1)$-vertex complete bipartite graph minus a matching of size $n-1$ (left) with $n$ colours that is not a fall colouring, and an example of a fall colouring of a $2n$-vertex complete bipartite graph minus a perfect matching (right) with $n$ colours.}\label{f-examples}
\end{figure}

The b-chromatic number was introduced by Irving and Manlove in 1999~\cite{IM99} as the worst- case behaviour of a natural heuristic for approximating the chromatic number of a graph $G$.
Starting with an arbitrary colouring, we iteratively try to improve it:
at each step, pick a colour $i$ and try
to reduce the number of colours by recolouring
each vertex $v$ of colour $i$ by a colour of $G$ that is not in $v$'s neighbourhood.  If such a recolouring is not possible for colour $i$, try another colour.  The process terminates when every colour has a b-chromatic vertex, and also shows that every graph has a b-colouring.  Other applications of the b-chromatic number arise in
graph clustering~\cite{effantin2006distributed}, postal mail sorting~\cite{gaceb2008improvement} and analysing hospital stays~\cite{elghazel2006new}.

In this paper we study colourings with b-chromatic vertices: in particular, we consider the three natural cases where each colour class has (i) at least one b-chromatic vertex, (ii) exactly one b-chromatic vertex, and (iii) all of its vertices being b-chromatic. From Case (i) we derive the {\sc b-Chromatic Number} problem, which is the problem of finding the b-chromatic number of a given graph. Case (ii) gives rise to
the {\sc Tight b-Chromatic Number}~\cite{HSS12,KM19,LS09}, which is the restriction of
{\sc b-Chromatic Number} to instances where the target number of colours is the maximum degree of the input graph and all maximum-degree vertices must become b-chromatic.
Case (iii) is precisely the definition of a \emph{fall colouring} of a graph~\cite{dunbar2000fall}.

A fall colouring is a b-colouring, but the converse is not necessarily true; see Figure~\ref{f-examples} for an example.
In fact, many graphs have no fall colouring; an example being the \emph{paw}, the triangle with a pendant vertex.
The problems {\sc Fall Chromatic Number} and {\sc Fall Achromatic Number}
are to find
the minimum and maximum integers $k$ for which a given graph $G$ admits a fall colouring with $k$ colours (or $0$ if $G$ has no fall colouring).

\medskip
\noindent
{\bf Our focus.}
Forbidden subgraphs have been widely studied from both a structural and algorithmic perspective~\cite[Chapter 7]{BLS99}, and in the area of graph colouring in particular~\cite{GJPS17}.
Given two graphs $G$ and $H$, we say that~$G$ is \emph{$H$-free} if $G$ does not contain $H$ as an \emph{induced} subgraph.
To increase our understanding of an \NP-hard problem~$\Pi$, we may aim to identify graphs $H$ for which $\Pi$ is  polynomial-time solvable or \NP-hard in $H$-free graphs, ideally
to obtain a complete dichotomy for the computational complexity of $\Pi$ in $H$-free graphs.
This has been done for several well-known colouring problems, such as
{\sc Chromatic Number}~\cite{KKTW01}, {\sc Precolouring Extension}~\cite{GPS14}, {\sc $\ell$-List Colouring}~\cite{GPS14},
and apart from one missing case, for {\sc Acyclic Colouring}, {\sc Star Colouring} and {\sc Injective Colouring}~\cite{BJMOPS25}.
We continue this line of study:

\medskip
\noindent
{\it Can we classify the complexity of
{\sc b-Chromatic Number}, {\sc Tight b-Chromatic Number}, {\sc Fall Chromatic Number} and {\sc Fall Achromatic Number} in $H$-free graphs?}

\medskip
\noindent
As we shall see, our systematic approach will allow us to unify and extend known results, and to
identify general patterns of complexity behaviour and new interesting open cases.

\medskip
\noindent
{\bf Terminology and bounds.}
Before discussing known complexity
results and our new results, we must
first introduce some new
terminology and basic observations.
A fundamental upper bound for the b-chromatic number involves the \emph{$m$-degree} of a graph $G$~\cite{IM99}, denoted by $m(G)$, which is the largest integer~$k$ such that $G$ has at least $k$ vertices with degree at least $k-1$.
Let $\chi(G)$ and $\varphi(G)$ denote the chromatic and b-chromatic number, respectively, of $G$. Then we observe:

\begin{observation}[\cite{IM99}]\label{o-0}
For a graph $G$, it holds that
$\chi(G)\leq \varphi(G)\leq m(G)$.
\end{observation}

\noindent
A vertex of $G$ with degree at least $m(G)-1$ is said to be {\it dense}, and $G$ is {\it tight}~\cite{HSS12,LS09} (referred to as \emph{$m$-tight} in~\cite{LS09}; see also~\cite{KZ17,LC13}) if $G$ has exactly $m(G)$ dense vertices, each of which has degree (exactly) $m(G)-1$. For example, the $4$-vertex cycle $C_4$ is not tight, but the triangle $C_3$ is tight (both graphs have $m$-degree $3$). This leads to the {\sc Tight b-Chromatic Number} problem,
formally introduced by Havet, Sales and Sampaio~\cite{HSS12},
which is to decide if a tight graph $G$ has a {\it tight} $b$-colouring, which is a $b$-colouring with $m(G)$ colours.

The \emph{fall spectrum} $\mathcal{F}(G)$ of a graph $G$ (also called the {\it fall set} of $G$) is the set of integers~$k$ for which $G$ admits a \emph{fall $k$-colouring}, which is a fall colouring with exactly $k$ colours. The {\it fall chromatic number}, denoted by
$\chi_f(G)$, is the smallest integer in $\mathcal{F}(G)$.
The {\it fall achromatic number}, denoted by $\varphi_f(G)$, is the largest integer in $\mathcal{F}(G)$.
We make the following known observation (see e.g.~\cite{Si19}):

\begin{observation}\label{o-known}
For a graph $G$ with $\mathcal{F}(G)\neq \emptyset$, it holds that $\chi(G)\leq \chi_f(G)\leq \varphi_f(G)\leq \delta(G)+1$, where $\delta(G)$ is the minimum degree of a vertex in $G$.
\end{observation}

\subsection{Known Complexity Results}\label{s-known}

The problems {\sc b-Chromatic Number}, {\sc Fall Chromatic Number} and {\sc Fall Achromatic Number} are all \NP-hard.
For {\sc b-Chromatic Number}, this was originally shown in~\cite{IM99}.
Later, it was proven that {\sc Tight b-Chromatic Number}, and thus {\sc b-Chromatic Number},
are \NP-hard for line graphs~\cite{CLMSSS15}\footnote{Campos et al.~\cite{CLMSSS15} state this result for
{\sc b-Chromatic Number}, but as we explain in Section~\ref{s-tight} their proof holds in fact for {\sc Tight b-Chromatic Number}.}
bipartite graphs~\cite{KTV02,Ma98} and for chordal graphs, or more precisely, disjoint unions of a connected split graph and three stars~\cite{HSS12}. In addition, {\sc b-Chromatic Number} is also \NP-hard for co-bipartite graphs~\cite{BSSV15}.
In contrast, {\sc Tight b-Chromatic Number} is polynomial-time solvable for co-bipartite graphs, block graphs and $P_4$-sparse graphs~\cite{HSS12}. The last result was known for {\sc b-Chromatic Number}~\cite{BDMMV09}, but tight $P_4$-sparse graphs allow for a faster algorithm~\cite{HSS12}. Later, it was shown that {\sc b-Chromatic Number} is polynomial-time solvable for graphs of girth at least~$7$~\cite{CLS15},
claw-free block graphs~\cite{CS18},
and for a superclass of
$P_4$-sparse graphs: $P_4$-tidy graphs~\cite{VBK11}. We note that $P_4$-tidy graphs have bounded clique-width~\cite{CMR00}. More recently,
Jaffke, Lima and Lokshtanov~\cite{JLL24} proved that
{\sc b-Chromatic Number} is polynomial-time solvable for every graph class of bounded clique-width~\cite{JLL24}, by showing that this holds even for {\sc b-Colouring}, which is the problem to decide if a graph has a b-colouring with exactly $k$ colours for some given integer~$k$.
Jaffke and Lima~\cite{JL20} established a range of complexity results for {\sc b-Colouring}.

Jaffke, Lima and Lokshtanov~\cite{JLL24} also proved that {\sc Fall Colouring}, which is to decide if a graph has a fall $k$-colouring for some given integer~$k$, is polynomial-time solvable for graphs of bounded clique-width. Consequently, the same holds for {\sc Fall Chromatic Number} and {\sc Fall Achromatic Number}. Their
result generalised results of
Silva~\cite{Si19} for $P_4$-sparse graphs
and Mitillos~\cite{Mi16} for threshold graphs.
The {\sc Fall Colouring} problem, and therefore {\sc Fall Chromatic Number} and {\sc Fall Achromatic Number}, can also be solved in polynomial time for split graphs~\cite{Mi16} and strongly chordal graphs~\cite{LDL05}. It was shown that for a graph~$G$
that is threshold, split or strong chordal, $G$ has a fall colouring if and only if $\chi(G)=\delta(G)+1$, and thus by Observation~\ref{o-known}, ${\cal F}(G)=\{\chi(G)\}$, or equivalently, $\chi_f(G)=\varphi_f(G)=\chi(G)$. We say that graphs whose fall spectrum consists of a singleton are {\it fall-unique}.  Silva~\cite{Si19} proved that chordal graphs with a fall colouring are fall-unique, but that
deciding non-emptiness of the fall spectrum is \NP-complete for chordal graphs.

If the number of colours $k$ is a fixed constant, then we obtain {\sc Fall $k$-Colouring}.
Already in 1998, Heggernes and Telle~\cite{HT98} proved that for every fixed $k\geq 3$, {\sc Fall $k$-Colouring} is \NP-complete. This was well before the notion of a fall colouring was introduced by Dunbar et al.~\cite{dunbar2000fall}. The reason is that a graph $G=(V,E)$ has a fall $k$-colouring if and only if $V$ can be partitioned into $k$
independent dominating sets (or equivalently, $k$
maximal independent sets). As such, {\sc Fall $k$-Colouring} belongs to the widely studied framework of locally checkable vertex partitioning problems introduced by Telle~\cite{Te93,Te94}.
Subsequently, Laskar and Lyle~\cite{LL09} proved that {\sc Fall $3$-Colouring} is \NP-complete for bipartite graphs, which was improved to planar bipartite graphs by
Lauri and Mitillos~\cite{LM20}. The latter authors~\cite{LM20} also proved that for every $k\geq 3$, {\sc Fall $k$-Colouring} is \NP-complete for $(2k-2)$-regular line graphs.\footnote{Theorem~15 in~\cite{LM20} is stated only for $(2k-2)$-regular graphs, but holds in fact for $(2k-2)$-regular line graphs, as explicitly stated in the first paragraph of the proof.}
As we will show later, all the above hardness results can be modified to hold for {\sc Fall Chromatic Number} and {\sc Fall Achromatic Number} as well by a simple trick.

\subsection{Our Results}\label{s-ours}
In our theorems, we let $G_1+G_2=(V(G_1)\cup V(G_2),E(G_1)\cup E(G_2))$ denote the  disjoint union of two
vertex-disjoint graphs, and we let $rG$ denote the disjoint union of $r$ copies of~$G$.
We write $G_1\ssi G_2$ if $G_1$ is an induced subgraph of $G_2$, and let $P_r$ be the path on $r$ vertices.

We can now state a dichotomy for {\sc b-Chromatic Number} in $H$-free graphs. We obtain this dichotomy
using
the aforementioned results from~\cite{BDMMV09,BSSV15,KTV02,Ma98}
after observing that the hardness gadget of Bonomo et al.~\cite{BSSV15} is not only co-bipartite but also $2P_2$-free;
see Section~\ref{s-b} for the details.

\begin{restatable}{theorem}{tb}\label{t-b}
For a graph~$H$, {\sc b-Chromatic Number} in $H$-free graphs is polynomial-time solvable if $H\ssi P_4$, and \NP-hard otherwise.
\end{restatable}

\noindent
It is well known that for a graph $H$, the class of $H$-free graphs has bounded clique-width if and only if $H\ssi P_4$ (see e.g.~\cite{DP16}). Hence, an alternative formulation of Theorem~\ref{t-b} is that subject to  $\sP\neq \NP$, {\sc b-Chromatic Number} is polynomial-time solvable in $H$-free graphs if and only if the class of $H$-free graphs has bounded clique-width.  Our next result shows that for tight graphs this correspondence with clique-width no longer holds. We prove it in Section~\ref{s-tight} by combining previous results from~\cite{BDMMV09,CLMSSS15,HSS12,KTV02,Ma98} with two new polynomial-time algorithms and two new \NP-completeness results.

\begin{restatable}{theorem}{ttight}\label{t-tight}
For a graph~$H$, {\sc Tight b-Chromatic Number} in $H$-free graphs~is
\begin{itemize}
\item polynomial-time solvable if $H\ssi P_4$,  $H\ssi P_3+P_1$ or $H\ssi 2P_2+P_1$
and
\item \NP-complete if $H$ is not a linear forest, or $P_5\ssi H$, $2P_3\ssi H$ or $3P_2\ssi H$.
\end{itemize}
\end{restatable}

\medskip
\noindent
Theorem~\ref{t-tight} shows, together with Theorem~\ref{t-b}, that there exist graphs $H$, namely
$H=P_3+P_1$, $H=2P_2$, $H=2P_2+P_1$,
$H=P_2+2P_1$ and $H=3P_1$, for which {\sc b-Chromatic Number} in $H$-free graphs is \NP-hard, while {\sc Tight b-Chromatic Number} is polynomial-time solvable. This difference in complexity for $H$-free graphs
was not known before. All the new polynomial-time cases in Theorem~\ref{t-tight} are for graph classes of unbounded clique-width. They generalise the aforementioned result for co-bipartite graphs~\cite{HSS12}, which are $3P_1$-free.
See Section~\ref{s-con} for a discussion on the missing cases.

Our final result, Theorem~\ref{t-fall}, shown in Section~\ref{s-fall}, consists of two coinciding dichotomies for fall colourings, which differ from the dichotomy in Theorem~\ref{t-b} and the partial dichotomy in Theorem~\ref{t-tight}. To prove it, we combine
the above hardness results from~\cite{LL09,LM20,Si19} with a new hardness result  for $(C_5,2P_2,P_2+2P_1,4P_1)$-free graphs and a new polynomial-time solvability result for $(P_3+P_1)$-free graphs.

\begin{restatable}{theorem}{tfall}\label{t-fall}
For a graph~$H$, {\sc Fall Chromatic Number} and {\sc Fall Achromatic Number} in $H$-free graphs are polynomial-time solvable if $H\ssi P_4$ or $P_3+P_1$, and \NP-hard otherwise.
\end{restatable}

\medskip
\noindent
The two dichotomies in Theorem~\ref{t-fall} also coincide with the dichotomy for {\sc Chromatic Number} for $H$-free graphs~\cite{KKTW01}. However, {\sc Chromatic Number} is polynomial-time solvable for other graph classes, such as chordal graphs, for which {\sc Fall Chromatic Number} and {\sc Fall Achromatic Number} are \NP-complete~\cite{Si19}.

\medskip
\noindent
{\bf Methodology.}
The new polynomial-time case in Theorem~\ref{t-fall} follows by using a result of Olariu~\cite{Ol88}.
However, in order
to prove the new polynomial-time results in Theorem \ref{t-tight}, we could not make use of previous tools, such as
the concepts of \emph{b-closure} and \emph{partial b-closure} as defined and used by Havet, Sales and Sampaio~\cite{HSS12}
to obtain polynomial-time algorithms for block graphs, co-bipartite graphs and $P_4$-sparse graphs.
The b-closure and partial b-closure of a graph are obtained by adding edges within the set of dense vertices
and/or within the set of their neighbours, in order to reduce to polynomial-time cases of
{\sc Chromatic Number} or {\sc Precolouring Extension}.
However, adding new edges to the input graph~$G$ may destroy the $H$-free property of $G$.
Moreover, {\sc Chromatic Number} and {\sc Precolouring Extension} are \NP-hard if $H$ is an induced subgraph of neither $P_1+P_3$ nor $P_4$~\cite{KKTW01}.
To circumvent this issue, we still exploit the idea behind precolouring extension and introduce the new technique of {\it tight $b$-precolouring extension}.
Namely, we first construct a partial colouring~$c'$ of the input graph~$G$ that satisfies some additional properties.
Afterwards, we will then try to extend $c$ to a tight $b$-colouring of~$G$.

\section{Preliminaries}\label{s-pre}

Let $G=(V,E)$ be a graph. For a set $S\subseteq V$, 
we let $G[S]$ denote the subgraph of $G$ induced by $S$, and 
we write $G-S=G[V\setminus S]$.
The {\it neighbourhood} of a vertex $u\in V$ is the set $N(u)=\{v\; |\; uv\in E\}$.
The \emph{degree} of $u$, denoted by $d(u)$, is equal to $|N(u)|$, and $G$ is {\it $r$-regular} if every vertex of $G$ has degree~$r$; if $r=3$, then $G$ is said to be {\it cubic}. For a set $S\subseteq V$, the \emph{outer boundary of $S$}~\cite{Ben13} is defined as 
$\delta S=(\bigcup_{u \in S}N(u))\setminus S$.  Two disjoint vertex subsets $S$ and $T$ in $G$ are {\it complete} to each other if 
every vertex of $S$ is adjacent to every vertex of $T$.

An {\it independent set} in 
$G$ is a set $I$ of pairwise non-adjacent vertices, which is {\it maximal} if $I\cup \{v\}$ is not independent for any $v\in V\setminus I$.  A {\it clique} in $G$ is a set of pairwise adjacent vertices. The {\it clique number} $\omega(G)$ of $G$ is the size of a largest clique in $G$.
A set $D\subseteq V$ is {\it dominating} if every vertex in $V\setminus D$ has at least one neighbour in $D$. If $D=\{u\}$ for some $u\in V$, then $u$ is a {\it dominating} vertex of $G$.
A {\it matching} in $G$
is a set $M\subseteq E$ such that no two distinct edges in $M$ share a common end-vertex,
which is {\it maximal} if $M\cup \{e\}$ is not maximal for any $e\in E\setminus M$.
A matching $M$ is {\it perfect} if every vertex of $G$ is incident to an edge of $M$.

Recall that a colouring of $G$ is a mapping $c:V\to \mathbb{Z}^+$ such that $c(u)\neq c(v)$ for every edge $uv\in E$. We say that $c(u)$ is the {\it colour} of $u$ and that all vertices mapped by $c$ to some specific colour~$i$ form a \emph{colour class}. 
Note that the colour classes of $c$ are independent sets that partition $V$. 

We recall that a graph $G$ is $H$-free for some graph~$H$ if $G$ does not contain $H$ as an induced subgraph. A graph $G$ is {\it $(H_1,\ldots,H_r)$-free} for some set of graphs $\{H_1,\ldots,H_r\}$ if $G$ is $H_i$-free for $i\in \{1,\ldots,r\}$. 

A {\it forest} is a graph with no cycles. A forest is {\it linear} if all its connected components are paths.
We let $C_r$ and $K_r$ denote the cycle and complete graph, respectively, on $r$ vertices. We also let $K_{1,r}$ denote the star on $r+1$ vertices; the graph $K_{1,3}$ is also known as the {\it claw}. A {\it split graph} is a graph whose vertex set can be partitioned into a clique $K$ and independent set $I$ (where we allow $K$ or $I$ to be empty). 
It is well known and not difficult to see that a graph is split if and only if it is $(C_4,C_5,2P_2)$-free.
A graph is {\it chordal} if it has no induced cycles of length greater than~$3$, that is, if it is $(C_4,C_5,\ldots)$-free.

A graph $G$ is {\it $P_4$-sparse} if every set of five vertices in $G$ induces at most one $P_4$. Note that $P_4$-free graphs are $P_4$-sparse.
The {\it line graph} $L(G)$ of a graph $G=(V,E)$ has as vertices the edges from $E$, and there exists an edge between two distinct vertices $e$ and $f$ of $L(G)$ if and only if $e$ and $f$ share a common end-vertex in $G$. It is well known and readily seen that every line graph is claw-free.

A graph $G=(V,E)$ is {\it complete multi-partite} if $V$ can be partitioned into independent sets $V_1,\ldots,V_r$ for some $r\geq 1$ such that for every $i,j$ with $i\neq j$, there exists an edge between every vertex of
$V_i$ and every vertex of $V_j$.
Recall that the paw is the graph on vertices $u_1,u_2,u_3,v_1$ and edges $u_1u_2$, $u_2u_3$, $u_3u_1$, $u_1v_1$. We will need the characterization of paw-free graphs due to Olariu~\cite{Ol88}.

\begin{lemma}[\cite{Ol88}]\label{l-olariu}
A graph $G$ is paw-free if and only if each connected component of $G$ is $C_3$-free or complete multi-partite. 
\end{lemma}

\noindent
We let $\overline{G}$ denote the {\it complement} of a graph $G$, which is obtained from $G$ by replacing each edge in $G$ with a non-edge, and vice versa. 
A {\it co-component} of $G$ is a (connected) component of $\overline{G}$, and a {\it co-bipartite} graph is the complement of a bipartite graph.
Note that $\overline{C_3}=3P_1$. Hence, as bipartite graphs are $C_3$-free, co-bipartite graphs are $3P_1$-free. Note also that
$\overline{\mbox{paw}}=P_3+P_1$. In fact, we will always use the ``complement'' of Lemma~\ref{l-olariu}:

\begin{corollary}[\cite{Ol88}]\label{c-olariu}
A graph $G$ is $(P_3+P_1)$-free if and only if each co-component of $G$ is $3P_1$-free or a disjoint union of complete graphs.
\end{corollary}

\noindent
A {\it clique cover} of a graph $G$ is a set of cliques ${\cal K}$ in~$G$ with the property that each vertex of~$G$ belongs to exactly one clique of ${\cal K}$. 
The {\it clique covering number}~$\sigma(G)$ is the size of a smallest clique cover of $G$. We observe that $\chi(G)=\sigma(\overline{G})$.

\section{The Proof of Theorem~\ref{t-b}}\label{s-b}

In this section we show Theorem~\ref{t-b} by a straightforward combination of the following known results, which we will use as lemmas. For the first one, we must however make the additional observation that the construction of Bonomo et al.~\cite{BSSV15}, used to show \NP-hardness of {\sc b-Chromatic Number} in co-bipartite graphs, is also $2P_2$-free.

\begin{figure}[b]
    \centering
    \includegraphics[width=0.3\linewidth]{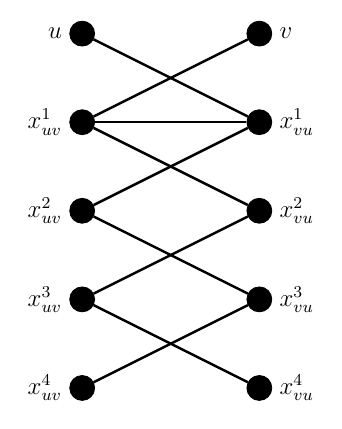}
    \caption{The gadget $H_{uv}$ used in the proof of Lemma~\ref{l-3p1}.}
    \label{fig:gadget-h-npc-2P2-free}
\end{figure}

\begin{lemma}[\cite{BSSV15}]\label{l-3p1}
{\sc b-Chromatic Number} is \NP-hard for $2P_2$-free co-bipartite graphs, and thus for $(3P_1,2P_2)$-free graphs.
\end{lemma}

\begin{proof}
The reduction given by Bonomo et al.~\cite{BSSV15} is based on replacing every edge $uv$ of a given bipartite graph $G$ by a gadget $H_{uv}$ with vertex set $\{u,v,x_{uv}^1, x_{uv}^2,x_{uv}^3, x_{uv}^4, x_{vu}^1,x_{vu}^2,x_{vu}^3,x_{vu}^4\}$ and
edge set $\{ux_{vu}^1,$ $vx_{uv}^1,$ $x_{uv}^1x_{vu}^1,$ $x_{uv}^1x_{vu}^2,$ $x_{vu}^1x_{uv}^2,$ $x_{uv}^2x_{vu}^3,$ $x_{vu}^2x_{uv}^3,$ $x_{uv}^3x_{vu}^4,$ $x_{vu}^3x_{uv}^4\}$; 
see also Figure \ref{fig:gadget-h-npc-2P2-free}. It may be verified that $H_{uv}$ is bipartite and $C_4$-free. Thus, $H=\bigcup_{uv\in E(G)}H_{uv}$ is bipartite and $C_4$-free.  Bonomo et al.'s reduction starts from {\sc Minimum Maximal Matching}~\cite{yannakakis1980edge} and involves showing a relationship between maximal matchings in $G$ and b-colourings in $\overline{H}$; hence $\overline{H}$ is the constructed instance of {\sc b-Chromatic Number}.  As $H$ is bipartite and $C_4$-free, it follows that  $\overline{H}$ is $(3P_1, 2P_2)$-free.  
\end{proof}

\noindent
The following two hardness results for {\sc Tight b-Chromatic Number} carry over to {\sc b-Chromatic Number}, as the former is a special case of the latter.
\begin{lemma}[\cite{KTV02,Ma98}]\label{l-bip}
{\sc Tight b-Chromatic Number} is \NP-complete for bipartite graphs, and thus for $C_3$-free graphs.
\end{lemma}

\begin{lemma}[\cite{HSS12}]\label{l-split}
{\sc Tight b-Chromatic Number} is \NP-complete for disjoint unions of a connected split graph and three stars, and thus for $(C_4,C_5,P_5)$-free graphs.
\end{lemma}

\noindent
We need one final known result.

\begin{lemma}[\cite{BDMMV09}]\label{l-p4}
{\sc b-Chromatic Number} is polynomial-time solvable for $P_4$-sparse graphs, and thus for $P_4$-free graphs.
\end{lemma}

\noindent
We are now ready to show Theorem~\ref{t-b}, which we restate below.

\tb*

\begin{proof}
Let $H$ be a graph. 
Suppose $H$ has an induced $C_r$.
If $r=3$, then we use Lemma~\ref{l-bip}.
If $r\geq 4$, then we use Lemma~\ref{l-split}.
Now assume $H$ is $(C_3,C_4,\ldots)$-free, so $H$ is a~forest.

If $H$ has an induced $2P_2$ or induced $3P_1$, then we use Lemma~\ref{l-3p1}.
Now assume $H$ is a $(3P_1,2P_2)$-free forest, or in other words, a $2P_2$-free linear forest with at most two connected components. 
First suppose  $H$ has two connected components. As $H$ is a $2P_2$-free linear forest, we find that $H=P_r+P_1$ for some $r\geq 1$. As $H$ is $3P_1$-free, we have that $r\leq 2$. Hence, $H=P_2+P_1$ or $H=2P_1$, so $H\ssi P_4$ and we can apply Lemma~\ref{l-p4}. 
Now suppose that $H$ is connected. Hence, as $H$ is a $2P_2$-free linear forest, $H=P_r$ for some $r\leq 4$, and we use Lemma~\ref{l-p4} again.
\end{proof}

\section{The Proof of Theorem~\ref{t-tight}}\label{s-tight}

In this section, we prove Theorem~\ref{t-tight}.
We will repeatedly make use of the following proposition, which follows from the definitions of the concepts involved. 

\begin{proposition}\label{p-dense}
Let~$G$ be a tight graph, and let~$T$ be the set of dense vertices of~$G$.
    In every tight b-colouring~$c$ of~$G$ (if one exists), the following holds:\\[-10pt]
    \begin{itemize}
    \item [(i)] the set $T$ is exactly the set of b-chromatic vertices under~$c$; and
    \item [(ii)] each colour class of $c$ has exactly one vertex of $T$, and this vertex has a unique neighbour in each of the other colour classes of $c$.
    \item [(iii)] every vertex of $T_1=\{v\in T : v\mbox{ is a dominating vertex of }G[T]\}$ belongs to a colour class of $c$ that does not contain any vertex of $\delta T$.
    \end{itemize}
\end{proposition}

\begin{proof}
We write $m=m(G)$. Let $T=\{u_1,\ldots,u_m\}$.
    Assume that~$G$ has a tight b-colouring~$c$.
    Let $V_1,\ldots, V_m$ be the colour classes of $c$.
    
We first show (i).
    Since~$c$ is tight, that is, uses~$m$ colours, every b-chromatic vertex has degree at least $m-1$, so is contained in~$T$.
    As~$c$ is a b-colouring, $c$ has at least~$m$ b-chromatic vertices.
    As~$G$ is tight, we have $|T|=m$.
    Therefore, $T$ is precisely the set of b-chromatic vertices.

 We now show (ii). Since $m=|T|$, each $V_i$ ($1\leq i\leq m$) has exactly one vertex from~$T$. Hence, we may assume without loss of generality that $T\cap V_i=\{u_i\}$ for each $i$ ($1\leq i\leq m)$.
   As every vertex $u_i$  ($1\leq i\leq m$) has degree $m-1$ and is b-chromatic, every  $u_i$ ($1\leq i\leq m$) has a unique neighbour in $V_j$ for each $j$ ($1\leq j\leq m$, $j\neq i$).

Finally, we show (iii). For a contradiction, assume that $u\in T_1$ and $s\in \delta T$ belong to the same colour class of $c$. This implies in particular that $u$ and $s$ are not adjacent. As $s\in \delta T$, there exists a vertex $u'\in T\setminus \{u\}$ that is adjacent to $s$. As $u$ is a dominating vertex of $T$, we also find that $u'$ is adjacent to $u$. Hence, $u'\in T$ has two neighbours in the same colour class, which contradicts (ii).
\end{proof}

\noindent
We now introduce two of our new key concepts, which are related to each other.

\begin{definition}
\label{d-precoldefs}
Let $G$ be a tight graph with a set $T$ of dense vertices. 
For a subset $S'\subseteq \delta T$, an
\emph{$S'$-partial b-colouring} of $G$ is a colouring~$c'$ of $G[S'\cup T]$ if:
\begin{enumerate}
\item $c'(u)\neq c'(u')$ for every two distinct vertices $u,u'\in T$, and
\item for every $u\in T$ with $c'(u)=c'(s)$ for some $s\in S'$, it holds that every vertex of $T\setminus \{u\}$ has exactly one neighbour with colour $c'(u)$.
\end{enumerate}
\end{definition}

\begin{definition}\label{d-extension}
Let $G$ be a tight graph with a set $T$ of dense vertices.
A b-colouring $c$ of $G$ is a \emph{b-precolouring extension} of an $S'$-partial b-colouring $c'$ in $G$ if:
\begin{enumerate}
\item $c(v)=c'(v)$ for every $v\in S'\cup T$ (that is, $c$ is an extension of $c'$), and
\item every colour class of $c$ with at least two vertices of $\delta T$ contains no vertex of $\delta T\setminus S'$.
\end{enumerate}
\end{definition}

\noindent
An $S'$-partial b-colouring $c'$ and a b-precolouring extension of $c'$ are illustrated for a tight graph in Figure \ref{fig:partial-precol}.
\begin{figure}
    \centering
    \includegraphics[width=0.6\linewidth]{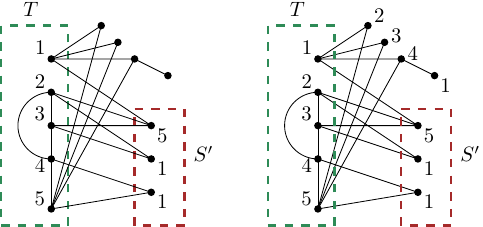}
    \caption{On the left: a tight graph $G$ with a set $T$ of dense vertices, where $G$ is annotated with an $S'$-partial b-colouring $c'$ (as per Definition~\ref{d-precoldefs}) using colours $\{1,\ldots,5\}$. On the right: a b-precolouring extension of $c'$ (as per Definition \ref{d-extension}).}
    \label{fig:partial-precol}
\end{figure}
The following observation is a straightforward corollary of Definition~\ref{d-precoldefs}.
\begin{observation}
\label{o-precolext}
Let $G$ be a tight graph.  If $c$ is a b-precolouring extension of an S-partial b-colouring $c'$ in $G$, then $c$ is a tight b-colouring of $G$.
\end{observation}
\begin{proof}
Recall that $|T|=m(G)$. Moreover, by definition, $c'(u)\neq c'(u')$ for every two distinct vertices $u,u'\in T$. Hence, $c'$ uses at least $m(G)$ colours.  Since $c$ is an extension of $c'$, this means that $c$ will use at least $m(G)$ colours.  Since $c$ is a b-colouring of $G$, we conclude that $c$ uses exactly $m(G)$ colours, so $c$ is tight.
\end{proof}

\noindent
In our next theorem we give a polynomial-time algorithm to determine whether an $S'$-partial b-colouring of a tight graph has a b-precolouring extension.

\begin{theorem}\label{t-semistrong}
Let $T$ be the set of dense vertices of a tight graph $G$, and let $S'\subseteq \delta T$, such that $G$ has an $S'$-partial b-colouring~$c'$.
It is possible to determine in polynomial time if $c'$ has a b-precolouring extension in $G$. \end{theorem}

\begin{proof}
Let $m=m(G)$, and let $T=\{u_1,\ldots,u_m\}$.  As $c'$ is an $S'$-partial b-colouring of $G$, we may assume that $c'(u_i)=i$ for every $i\in \{1,\ldots,m\}$.
We let $T'$ 
denote the set of 
vertices in $T$ that have the same colour as a vertex of $S'$. 
We let $T_1=\{v\in T : v\mbox{ is a dominating vertex of }G[T]\}$, so for every $u\in T_1$, every vertex of $T\setminus \{u\}$ is adjacent to~$u$. 

We claim that $T_1\cap T'=\emptyset$. For a contradiction, suppose that there exists a vertex $u\in T_1\cap T'$. Then, as $u\in T'$, there exists a vertex $s\in S'$ with colour $c'(u)$. Since $u$ and $s$ have the same colour, $u$ and $s$ are not adjacent. As $S'\subseteq \delta T$, this means that there exists a vertex $u'\in T\setminus \{u\}$ that is adjacent to $s$. However, now $u'$ is adjacent to both $u$ and $s$, contradicting the fact that $c'$ is an $S'$-partial b-colouring of $G$.

We set $T_2=T\setminus (T_1\cup T')$. Note that $T=T_1\cup T_2\cup T'$, and that the sets $T_1$, $T_2$ and $T'$ are pairwise disjoint. Moreover, we can determine $T_1$, $T_2$ and $T'$ in polynomial time.

\medskip
\noindent
First, we assume that $T_2=\emptyset$, so $T=T_1\cup T'$.  We check, in polynomial time, if $S'=\delta T$.

\medskip
\noindent
{\bf Case 1.}  $S'\neq \delta T$.\\
We claim that $c'$ has no b-precolouring extension in $G$.
For a contradiction, suppose that $c$ is a b-precolouring extension of $c'$ in $G$.
As $S'\subseteq \delta T$ and $S'\neq \delta T$, we find that $\delta T\setminus S'$ contains a vertex $s$. As $c$ is a b-precolouring extension, $c$ is a tight b-colouring of $G$ by Observation \ref{o-precolext}. Hence, from Proposition~\ref{p-dense}~(iii) it follows that $c(s)$ is not a colour of a vertex in $T_1$.
We recall that $1\leq c(s)\leq m$ and that by definition, $c'$ uses all $m$ colours on the $m$ vertices of $T$. Hence, $c(s)=c(u)$ for some vertex $u\in T'$. From the definition of $T'$ it follows that $c'(u)=c'(s')$ for some $s'\in S'$. However, as $c$ extends $c'$, we also have that $c(s)=c(u)=c'(u)=c'(s')=c(s')$. In other words, $c$ has a colour class with at least two vertices of $\delta T$, namely $s$ and $s'$, that contains a vertex, namely $s$, of $\delta T\setminus S'$. This contradicts our assumption that $c$ is a b-precolouring extension of $c'$ in $G$.

\medskip
\noindent
{\bf Case 2.} $S'=\delta T$.\\
Let $u\in T$. As $c'$ is an $S'$-partial b-colouring, $u$ has exactly one neighbour with colour $c'(u')$ for every $u'\in T'\setminus \{u\}$ by definition.
As $T_1\cap T'=\emptyset$ and every vertex of $T_1$ dominates $G[T]$, we also find that $u$ has exactly one neighbour with colour $c'(u')$ for every $u'\in T_1\setminus \{u\}$.
As $c'$ uses $m$ colours on $G[T]$, we conclude that $c'$ is a tight b-colouring of 
$G[T\cup S']=G[T\cup \delta T]$.

Due to the above, we can extend $c'$ to a 
b-precolouring extension $c$ in $G$ as follows. Set $c(u)=c'(u)$ for all $u\in T\cup \delta T$.
As every vertex $w\in V(G)\setminus (T\cup \delta T)$ is not dense, $w$ has degree at most~$m-2$ in $G$. Hence, we can greedily colour (in $c$) the vertices of $V(G)\setminus (T\cup \delta T)$ one by one. Since $\delta T\setminus S'=\emptyset$, the second property of Definition~\ref{d-extension} is now satisfied for $c$ as well.
Moreover, we found $c$ in polynomial time. 
\begin{figure}
\centering
\scalebox{0.55}{\begin{tikzpicture}[
    v_node/.style={circle, fill=black, minimum size=12pt, inner sep=0pt},
    w_node/.style={circle, fill=white, minimum size=12pt, inner sep=0pt, draw=black, line width=1.2pt},
    dot_node/.style={draw=none},
    edge_style/.style={line width=1.2pt},
    every label/.append style={font=\Large},
    custom_dotted/.style={dash pattern=on 0pt off 10pt, line cap=round},
]
    
    \node[v_node, red] (u1) {};
    \node[v_node, green, right=1cm of u1] (u2) {};
    
    \node[dot_node, right=0.5cm of u2] (dots) {\Huge$\cdots$};
    \node[v_node, blue, right=0.5cm of dots] (u3) {};
    \node[draw, line width=1.5pt, rounded corners, fit=(u1) (u3), inner sep=0.5cm, label={[xshift=-1cm]above left:$T_1$}] (sq1) {};

    \node[left=0.5cm of sq1] (Tlabel) {\Large $T$};
    \node[below=4cm of Tlabel] {\Large $\delta T$};

    \node[v_node, cyan, right=3cm of u3] (v1) {};
    \node[v_node, magenta, right=1cm of v1] (v2) {};
    \node[dot_node, right=0.4cm of v2] (vdots) {\Huge$\cdots$};
    \node[v_node, yellow, right=0.5cm of vdots] (v3) {};
    \node[draw, line width=1.5pt, rounded corners, fit=(v1) (v3), inner sep=0.5cm, label={above:$T'$}] (sq2) {};

    \node[v_node,orange, right=3cm of v3] (w1) {};
    \node[v_node,violet, right=0.5cm of w1] (w2) {};
    \node[dot_node, right=0.5cm of w2] (wdots) {\Huge$\cdots$};
    \node[v_node, brown, right=0.5cm of wdots] (w3) {};
    \node[draw, line width=1.5pt, rounded corners, fit=(w1) (w3), inner sep=0.5cm, label={[xshift=1cm]above right:$T_2$}] (sq3) {};

    \node[w_node, cyan, below=3cm of v1] (s1) {};
    \node[w_node, magenta, right=1cm of s1] (s2) {};
    \node[w_node, yellow, right=2cm of s2] (s3) {};
    \node[dot_node, below=0.25cm of s1, scale=1.5, transform shape] (s1dots) {\Huge$\vdots$};
    \node[dot_node, below=0.25cm of s2, scale=1.5, transform shape] (s2dots) {\Huge$\vdots$};
    \node[dot_node, below=0.25cm of s3, scale=1.5, transform shape] (s3dots) {\Huge$\vdots$};
    \node[w_node, cyan, below=0.5cm of s1dots] (s4) {};
    \node[w_node, magenta, below=0.5cm of s2dots] (s5) {};
    \node[w_node, yellow, below=0.5cm of s3dots] (s6) {};
    \node[draw, line width=1.5pt, rounded corners, fit=(s1) (s6), inner sep=0.5cm, label={left:$S'$}] (sq4) {};

    \node[w_node, orange, below=3cm of w1] (z1) {};
    \node[w_node, violet, right=0.5cm of z1] (z2) {};
    \node[dot_node, right=0.5cm of z2] (zdots) {\Huge$\cdots$};
    \node[w_node, brown, right=0.5cm of zdots] (z3) {};
    \node[draw, line width=1.5pt, rounded corners, fit=(z1) (z3), inner sep=0.5cm, label={right:$S$}] (sq5) {};

    \draw[edge_style, line width=3pt, bend left] (sq1) to (sq2);
    \draw[edge_style, line width=3pt, bend left=40] (sq1) to (sq3);
    \draw[edge_style, line width=1pt] (sq2.east) to (sq3.west);
    \draw[edge_style, line width=1pt] (sq2) to (sq4.north);
    \draw[edge_style, line width=1pt] (sq2) to (sq5);
    \draw[edge_style, line width=1pt] (sq3) to (sq4);
    \draw[edge_style, line width=1pt] (sq3) to (sq5.north);
    \draw[edge_style, line width=1pt] (sq3) to (sq5.north);
    \draw[edge_style, line width=1pt] (sq4) to (sq5);

\end{tikzpicture}}
\caption{The sets $T$ and $\delta T$ from the proof of Theorem \ref{t-semistrong}, where $T$  is the disjoint union of $T_1$, $T'$ and $T_2$, and $\delta T$ is the disjoint union of $S'$ and $S$. Thick edges indicate that the set $T_1$ is complete to both $T'$ and $T_2$.  
Thin edges between sets indicate edges that may or may not exist between vertex pairs. Note that $T_1$ is a clique and that no edges exist between $T_1$ and $\delta T$. All ``vertical layers'' are monochromatic and given a unique colour. That is, the $S'$-partial b-colouring $c'$ colours the vertices of $T\cup S'$, where each vertex in $T'$ has at least one non-neighbour in $S'$ of the same colour by definition. As displayed in the figure and shown in the proof, the vertices in $T_2$ and $S$ must be in one-to-one correspondence in any $b$-colouring extension of $G$ (if there exists~one).}
\label{f-semistrong}
\end{figure}

\medskip
\noindent
From now on, suppose that $T_2\neq \emptyset$. We let $S=\delta T\setminus S'$. At this point, we refer to Figure \ref{f-semistrong} for an illustration of the sets $T$ and $\delta T$. From $G$ we define a bipartite graph~$G^*$ as follows:

\begin{itemize}
\item let $T_2$ and $S$ be the two partition classes of $G^*$, and
\item add an edge between a vertex $u\in T_2$ and a vertex $s\in S$ if and only if $us\notin E(G)$ and moreover, $u$ and $s$ have no common neighbour in $(T_2\cup T')\setminus \{u\}$ in $G$.
\end{itemize}

\noindent
Note that constructing $G^*$ takes polynomial time. We now prove that $c'$ has a b-precolouring extension in $G$ if and only if $G^*$ has a perfect matching.

\medskip
\noindent
($\Rightarrow$)
Suppose that $c'$ has a b-precolouring extension $c$ in $G$. By Observation~\ref{o-precolext}, $c$ is a tight b-colouring of $G$.
 By Proposition~\ref{p-dense}~(i), the set~$T$ is exactly the set of b-chromatic vertices under~$c$. By Proposition~\ref{p-dense}~(ii), each colour class of $c$ contains exactly one vertex of $T$. We say that a colour class of $c$ is of {\it type~1} or {\it type~2} depending on whether its b-chromatic vertex from $T$ belongs to $T_1\cup T'$ or $T_2$, respectively.

By Proposition~\ref{p-dense}~(iii), no vertex of $S$ belongs to the same colour class as a vertex of $T_1$.
As $c$ is a b-precolouring extension of $c'$, no vertex of $S$ belong to the same colour class as a vertex of $T'$ either.
Hence, all vertices of $S$ are in colour classes of type~2.  As $T_2\subseteq T\setminus T_1$, every vertex $u\in T_2$ has at least one non-neighbour $u'\in T$.
As every vertex of $T_1$ is a dominating vertex of $G[T]$, we find that $u'\notin T_1$. 
Since $u'$ must still have a neighbour with colour $c(u)$, we find that $u'$ must be adjacent to a vertex $s\in \delta T$ with colour $c(u)$. As $u\in T_2$ and $T_2\cap T'=\emptyset$, we find that $s\in S$.
Since $c$ is a b-precolouring extension of $c$, 
every colour class of $c$ with at least two vertices of $\delta T$ contains no vertex of $S$. Hence, every colour class of type~2 contains exactly one vertex of $T_2$ and exactly one vertex of $S$. 

The above implies that $|S|=|T_2|$, and we may write $S=\{s_1,\ldots,s_{|T_2|}\}$ and $T_2=\{u_1,\ldots,u_{|T_2|}\}$.
Let $V_1,\ldots,V_{|T_2|}$ be the colour classes of type~2 of $G$. By symmetry we may now assume that $V_i\cap T_2=\{u_i\}$ and $V_i\cap S=\{s_i\}$ for every $i$ ($1\leq i\leq |T_2|$).

Let $1\leq i\leq |T_2|$, and let $u_j\in (T_2\cup T')\setminus \{u_i\}$.  As $u_j$ is b-chromatic, $d(u_j)=m-1$ and $c(u_i)=c(s_i)$, it follows that $u_j$ has exactly one neighbour in $\{u_i,s_i\}$. Hence, $u_i$ and $s_i$ have no common neighbour in $(T_2\cup T')\setminus \{u_i\}$. 
As a colour class is an independent set, $u_i$ and $s_i$ are not adjacent. As a consequence, $u_is_i$ is an edge of $G^*$. Moreover, as $G^*$ has partition classes $T_2$ and $S$, the set  $M=\{u_is_i : 1\leq i\leq |T_2|\}$ is a perfect matching of $G^*$. 

\medskip
\noindent
($\Leftarrow$) 
Suppose that $G^*$ has a perfect matching~$M$. As $G^*$ is bipartite with partition classes $S$ and $T_2$, this means that $|S|=|T_2|$. Hence, we may write $S=\{s_1,\ldots,s_{|T_2|}\}$ and $T_2=\{u_1,\ldots,u_{|T_2|}\}$ such that $M=\{u_is_i : 1\leq i\leq |T_2|\}$.

We assume without loss of generality that $c'(u_i)=i$ for every $i$ ($1\leq i\leq |T_2|$).
We first colour the vertices of $S$ by giving $s_i$ colour~$i$ for every $i$ ($1\leq i\leq |T_2|$).
As $T_2\cap T'=\emptyset$, no vertex in $S'$ has
the same colour as a vertex in $T_2$. Hence, no two adjacent vertices in $G[T\cup \delta T]$ are coloured the same, and moreover,
every vertex of $S$ belongs to a colour class of size~$2$, together with a vertex of $T_2$.
Since every vertex of $V(G)\setminus (T\cup \delta T)$ does not belong to $T$, it is not dense and thus has degree at most $m-2$. So, just as before, we can safely colour the vertices of $V(G)\setminus (T\cup \delta T)$ one by one in a greedy way.  Let $c$ be the resulting colouring of $G$.  We claim that $c$ is a b-precolouring extension of $c'$ in $G$.

By construction, $c$ is an extension of $c'$.
As each vertex of $S$ belongs to a colour class of size~2, together with a vertex of $T_2$, every colour class of $c$ with at least two vertices of $\delta T$ contains no vertex of $S=\delta T\setminus S'$.  This means that both properties of Definition~\ref{d-extension} are satisfied.
However, we must still show that $c$ is not only a colouring but also a b-colouring. 
In order to do this, we first note that, since $c'$ uses $m$ colours and we did not introduce any new colours, $c$ uses $m$ colours as well.
It remains to show that every vertex of $T$ is b-chromatic.

First, consider a vertex $u\in T_1$. As $u\in T_1\subseteq T$, it has degree $m-1$. By the definition of~$c'$, every vertex of $T$ belongs to a different colour class of $c'$, and thus to a different colour class of $c$. Since $T_1$ dominates $G[T]$, this means that $u$ has a neighbour in every colour class not its own. Hence, $u$ is b-chromatic.

Now, consider a vertex $u\in T_2\cup T'$. Since $T_1$ dominates $G[T]$ and $(T_2\cup T')\subseteq T$, we find that $u$ is adjacent to all vertices of $T_1$, each of which belongs to a different singleton colour class of~$c$ by construction. In other words, $u$ has exactly one neighbour with colour $c(u')$ for every $u'\in T_1$, namely $u'$.

Now suppose that $u'\in T'\setminus \{u\}$.  By definition of $T'$, $c'(u')=c'(s)$ for some $s\in S'$.  As $c'$ satisfies Property 2 of Definition \ref{d-precoldefs}, it follows that $u$ has exactly one neighbour with colour~$c'(u')$.  As $c$ is an extension of $c'$, we conclude that $u$ also has exactly one neighbour with colour $c(u')$ for every $u'\in T'\setminus \{u\}$.

It remains to show that $u$ is also picking up every colour used on $T_2\setminus \{u\}$ in $c$.
For a contradiction, suppose that $u$ is not adjacent to any vertex of $\{u_j,s_j\}$ for some $u_j\in T_2\setminus \{u\}$. 
As $c$ has $m$ colour classes and $u$ has degree $m-1$ in $G$, this means that $u$ has two neighhbours with the same colour.
Previously, we found that $u$ has exactly one neighbour with colour $c(u')$ for every $u'\in (T_1\cup T')\setminus \{u\}$.
Thus, there must exist a pair $\{u_h,s_h\}$
with $u_h\in T_2\setminus \{u,u_j\}$ such that
$u$ is adjacent to both $u_h$ and $s_h$, by the pigeonhole principle. That is, $u\in (T_2\cup T')\setminus \{u_h\}$ is a common neighbour of $u_h$ and $s_h$. However, it now follows from the construction of~$G^*$ that $u_hs_h$ is not an edge in $G^*$, and thus $u_hs_h$ cannot be an edge of $M$, a contradiction. We conclude that $u$ is b-chromatic.

\medskip
\noindent
It now suffices to check if~$G^*$ has a perfect matching, which  takes polynomial time~\cite{Ed65}.
\end{proof}

\noindent
We can now prove our polynomial-time result for $(2P_2+P_1)$-free graphs that we will use as a lemma in the proof of 
Theorem~\ref{t-tight}.
Recall that in contrast, {\sc b-Chromatic Number} is \NP-hard even for 
the smaller class of $3P_1$-free graphs by Lemma~\ref{l-3p1}~\cite{BSSV15}.

\begin{lemma}\label{l-tight2p2p1}
{\sc Tight b-Chromatic Number} is polynomial-time solvable for $(2P_2+P_1)$-free graphs.
\end{lemma}

\begin{proof}
Let $G=(V,E)$ be a tight $(2P_2+P_1)$-free graph. We write $m=m(G)$. Let $T=\{u_1,\ldots,u_m\}$ be the set of dense vertices of $G$, so $|T|=m$ and every $u_i$ $(1\leq i\leq m)$ is a maximum-degree vertex of $G$ with degree (exactly) $m-1$. 

\begin{figure}
    \centering
    \includegraphics[width=0.5\linewidth]{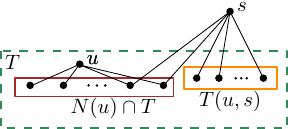}
    \caption{An illustration of a set $T$ of dense vertices containing the disjoint subsets $N(u) \cap T$ and $T(u,s)$, where $u\in T$ and $s\in S=\delta T$, as per the proof of Lemma \ref{l-tight2p2p1}.  Note that $u$ is not adjacent to~$s$, and by definition, no vertex in $T(u,s)$ is adjacent to $u$.}
    \label{fig:Lem-17-T(u,s)}
\end{figure}

We let $S=\delta T$ be the set of neighbours of vertices of $T$ that do not belong to $T$. For two non-adjacent vertices $s\in S$
and $u\in T$, we let $T(u,s)$ consist of all neighbours of $s$ in $T$ that are not adjacent to $u$ (see Figure~\ref{fig:Lem-17-T(u,s)}); note that $T(u,s)=\emptyset$ is possible.

We will begin by showing a structural claim. Assume that $G$ admits a tight b-colouring $c$. By
Proposition~\ref{p-dense}~(i), $T$ is the set of b-chromatic vertices under~$c$. Let $V_1,\ldots,V_m$ be the $m$ colour classes of $c$.  By Proposition~\ref{p-dense}~(ii), we may assume without loss of generality that $V_i\cap T = \{u_i\}$.

\begin{claim}\label{c-ui}
If $|V_i|\geq 3$ for some $i$ $(1\leq i\leq m)$ and $s\in S\cap V_i$, then $T(u_i,s)$ is complete to $T\setminus (T(u_i,s)\cup \{u_i\})$.
\end{claim}

\begin{claimproof}
Let $|V_i|\geq 3$ for some $i$ $(1\leq i\leq m)$, say $i=1$. Let $s\in S\cap V_1$. We note that
$s$ and $u_1$ are not adjacent,
since $V_1$ is an independent set containing $s$ and $u_1$, and thus, the set $T(u_1,s)$ 
is defined.
For a contradiction, assume that $T(u_1,s)$ contains a vertex, say $u_2$, that is not adjacent to some vertex, say $u_3$, in $T\setminus (T(u_1,s)\cup \{u_1\})$. 
We distinguish two cases.

First suppose that $u_3$ is not adjacent to $u_1$. Since $u_3\notin T(u_1,s)$ and $u_3$ is b-chromatic by Proposition~\ref{p-dense}~(i), this means that $u_3$ is adjacent to a vertex $s'\in S\cap V_1$. Since $V_1$ is an independent set, $u_1,s,s'$ are pairwise non-adjacent. 
By Proposition~\ref{p-dense}~(ii), we find that $u_2$ and $u_3$ have a unique neighbour in $V_1$. As $u_2$ and $u_3$ are adjacent to $s\in V_1$ and $s'\in V_1$, respectively, this means that $u_2$ is  adjacent neither to $u_1$ nor to $s'$, and $u_3$ is adjacent neither to $u_1$ nor $s$. However, now $\{u_2,s,u_3,s',u_1\}$ induces a $2P_2+P_1$ in $G$, a contradiction.

Now suppose that $u_3$ is adjacent to $u_1$. Since $V_1$ has size at least~$3$, there exists a vertex $s'\in V_1\setminus \{u,s\}$.
As $u_3$ has a unique neighbour in $V_1$ due to Proposition~\ref{p-dense}~(ii), and $u_3$ is adjacent to $u_1\in V_1$, we find that $u_3$ is adjacent neither to $s$ nor to $s'$.
As $u_2$ has a unique neighbour in $V_1$ for the same reason, and $u_2$ is adjacent to $s\in V_1$, we find that $u_2$ is adjacent neither to $u_1$ nor to $s'$.
As $V_1$ is an independent set, $u_1,s,s'$ are pairwise non-adjacent. However, now $\{u_1,u_3,u_2,s,s'\}$ induces a $2P_2+P_1$ in $G$, a contradiction.
\end{claimproof}

\noindent
We now present our polynomial-time algorithm, which consists of four steps.

\medskip
\noindent
{\bf Step 1.}
 We colour all the $m$ vertices of $T$ by giving them a unique colour from $\{1,\ldots,m\}$. 
 
 \medskip
 \noindent
 Step~1 takes polynomial time and is correct due to Proposition~\ref{p-dense}~(ii).

\medskip
\noindent
{\bf Step 2.} We check for each $u\in T$ and each $s\in S$ with $us\notin E(G)$, whether $T(u,s)$ is complete to $T\setminus (T(u,s)\cup \{u\})$. If so, we give $s$ the same colour as $u$. 

\medskip
\noindent
Step 2 takes polynomial time. It is correct for the following reasons. Suppose $T(u,s)$ is complete to $T\setminus (T(u,s)\cup \{u\})$ for some $u\in T$ and $s\in S$ with $us\notin E(G)$. We cannot give $s$ the same colour as a vertex in $T(u,s)$, because $s$ is adjacent to every vertex of $T(u,s)$ by definition. Moreover, we cannot give $s$  the same colour as a vertex in $T\setminus (T(u,s)\cup \{u\})$. Otherwise, as $T(u,s)$ is complete to $(T\setminus (T(u,s)\cup \{u\}))\cup \{s\}$, the vertices from $T(u,s)$ would be adjacent to two vertices of the same colour. As $T(u,s)\subseteq T$, this is not possible due to Proposition~\ref{p-dense}~(ii). Hence, as $s$ must receive a colour from $\{1,\ldots,m\}$, we have no other choice than to give $s$ the same colour as~$u$. Also, by the same reasoning, if the same vertex $s$ is coloured more than once during Step 2 then we return that $G$ has no tight b-colouring.

\medskip
\noindent
Let $S'$ be the subset of vertices of $S$ that we have already coloured. Let $c'$ be the colouring we constructed so far, so $c'$ is a colouring of $G[S'\cup T]$.
Note that we obtained both $S'$ and $c'$ in polynomial time. 

\medskip
\noindent
{\bf Step 3.} We check if $c'$ is an $S'$-partial b-colouring of $G$. If not, then we return that $G$ has no tight b-colouring.

\medskip
\noindent
Step~3 takes polynomial time, as from Step~1, we know that $c'(u)\neq c'(u')$ for every two distinct vertices $u,u'\in T$, and
it takes polynomial time to check if

\begin{itemize}
\item $c'$
is a colouring of $G[S'\cup T]$,
\item for every $u\in T$ with $c'(u)=c'(s)$ for some $s\in S'$, it holds that every vertex of $T\setminus \{u\}$ has exactly one neighbour with colour $c'(u)$.
\end{itemize}

\noindent
Step~3 is correct for the following reasons. First, we have already shown above that the b-colouring that we constructed so far is unique (modulo permutation of colours). Hence, if two adjacent vertices in $G[S'\cup T]$ received the same colour, then we may safely output that $G$ is a no-instance. By Proposition~\ref{p-dense}~(ii), we may do the same if there is a vertex in $T$ that is adjacent to at least two vertices with the same colour. 

Now suppose that there exists a vertex $u\in T$ with $c'(u)=c'(s)$ for some $s\in S'$, such that some vertex $u'\in T\setminus \{u\}$ is not adjacent to any vertex with colour~$c'(u)$. In particular, this means that $u$ and $u'$ are not adjacent. If $G$ still has a tight b-colouring, then $u'$ must be adjacent to some vertex $s^*\in S\setminus S'$ that will be given colour~$c'(u)$ later, due to Proposition~\ref{p-dense}~(i). Hence, every tight b-colouring of $G$ (if there exists one) will have at least three vertices with colour $c'(u)$. However, by Claim~\ref{c-ui}, $T(u,s^*)$ must be complete to $T\setminus (T(u,s^*)\cup \{u\})$. This means that $s^*\in S'$, a contradiction.

\medskip
\noindent
Suppose we have not returned a no-answer yet. 

\medskip
\noindent
{\bf Step 4.} We check if $G$ contains a b-precolouring extension~$c$ of $c'$. If so, then we return $c$, 
and otherwise we return that $G$ has no tight b-colouring.

\medskip
\noindent
Step~4 takes polynomial time by Theorem~\ref{t-semistrong} and is correct for the following reasons. 
For a contradiction, assume that $G$ has a tight
b-colouring $c^*$ 
that is not a b-precolouring extension of $c'$.
We already deduced that $c'$ is unique (modulo permutation of colours). Hence, we may assume without loss of generality that $c$ is an extension of $c'$. As $c^*$ is not a b-precolouring extension, this means that $c^*$ must have a colour class of size at least~$3$ that contains a vertex $s\in S\setminus S'$. Say the vertices in this colour class have colour~$i$ ($1\leq i\leq m$).
By Proposition~\ref{p-dense}~(ii), we find that $T$ contains exactly one vertex $u$ with $c^*(u)=i$.
However, as there are at least three vertices of colour~$i$, we find that $T(u,s)$ must be complete to $T\setminus (T(u,s)\cup \{u\})$ by Claim~\ref{c-ui}.
This means that $s\in S'$, a contradiction.

\medskip
\noindent
As we showed that Steps~1--4 each take polynomial time, our algorithm runs in polynomial time. Correctness follows from the correctness proofs given for each of these four steps.
\end{proof}

\noindent
We now give an algorithm for $(P_3+P_1)$-free graphs. In it, we use our algorithm for $(2P_2+P_1)$-free graphs as a subroutine. As such, also this result relies (indirectly) on Theorem~\ref{t-semistrong}.

\begin{lemma}\label{l-tightp13}
{\sc Tight b-Chromatic Number} is polynomial-time solvable for $(P_3+P_1)$-free graphs.
\end{lemma}

\begin{proof}
Let $G$ be a $(P_3+P_1)$-free tight graph with co-components $G_1, \ldots, G_r$ for some $r\geq 1$. We write $m=m(G)$ and $n=|V(G)|$.
Let $T$ be the set of dense vertices of $G$, so $|T|=m$ and every vertex in $T$ has degree $m-1$ in $G$. For each $i$ ($1\leq i\leq r$), we set $T_i =T \cap V(G_i)$, and write $m_i=|T_i|$ and $n_i=|V(G_i)|$, so
\[m=m_1+\cdots + m_r\; \mbox{and}\; n=n_1+\cdots+n_r.\]
By definition, $V(G_i)$ is complete to $V(G_j)$ if $i\neq j$.
As $G$ is tight, this means that
for each $i$ ($1\leq i\leq r$), the vertices in $T_i$ are exactly the maximum-degree vertices of $G_i$ and have the same degree in $G_i$, which we denote by $p_i$. Note that we can compute the sets~$T_i$ and degrees $p_i$ $(1\leq i\leq r)$ in polynomial time.

We now prove the following two claims.

\begin{claim}\label{cl-P3P1-1}
    It holds that $m_i\geq p_i+1$ for each $i$ ($1\leq i\leq r$).
\end{claim}
\begin{claimproof}
    We first show that $T_i\neq \emptyset$ for each $i$ ($1\leq i\leq r$). For a contradiction, say $T_1=\emptyset$. Let $u\in V(G_1)$. As $u$ is complete to $V(G_2)\cup \cdots \cup V(G_r)$ and $T\subseteq V(G_2)\cup \cdots \cup V(G_r)$, we find that $u$ has at least degree $m$ in $G$. This contradicts the fact that $G$ has maximum degree $m-1$.
    
    We will now show that $m_i\geq p_i+1$ for each $i$ ($1\leq i\leq r$). By symmetry, it suffices to show this for $i=1$. 
    Since $T_1\neq \emptyset$, there exists a vertex $u\in T_1$. Recall that $u$ has degree $m-1$ in $G$. As $u$ is complete to $V(G_2)\cup \cdots \cup  V(G_r)$, this implies that $m-1=p_1+n_2+\cdots +n_r$, or equivalently,
    \[|T|=m=p_1+1+n_2+ \cdots +n_r.\]
    As $T$ can contain at most $n_2+\cdots + n_r$ vertices of $V(G_2)\cup \cdots \cup V(G_r)$, this means that $T$ must have at least $p_1+1$ vertices in $G_1$, that is, $m_1\geq p_1+1$.
\end{claimproof}

\begin{claim}\label{cl-P3P1-2}
    If $m_i>p_i+1$ for some $i$ $(1\leq i\leq r)$, then $G$ has no tight b-colouring.
\end{claim}
\begin{claimproof}
    Suppose without loss of generality that $m_1>p_1+1$. Let $u\in T_1$. For a contradiction, suppose that $G$ has a tight 
    b-colouring $c$. By Proposition~\ref{p-dense}, the vertices of $T$ are b-chromatic and assigned different colours. As $m_1>p_1+1$, we now find that $u$ misses a colour $x$ of one of the vertices of $T_1$ in its neighbourhood in $G_1$. As $V(G_1)$ is complete to $V(G_2)\cup \cdots \cup V(G_r)$, we find that $x$ is not used on $V(G_2)\cup \cdots \cup V(G_r)$. Hence, $u$ misses $x$ even in its neighbourhood $N(u)$ in~$G$, contradicting the fact that $u\in T_1$ is b-chromatic under $c$.
\end{claimproof}

\noindent
We can check if $m_i>p_i+1$ for some 
$i$ $(1\leq i\leq r)$ in polynomial time. If so, we return that $G$ has no 
tight b-colouring.\footnote{An example where this happens is where $G$ is the tight $(P_3+P_1)$-free graph consisting of two co-components $G_1=2P_1$ and $G_2=K_3+P_1$. In this example, $T$ consists of $V(G_1)$ and the three vertices of the triangle in $G_2$, so $m=5$, $m_1=2$ and $m_2=3$, but $p_1=0$. It can be readily checked that $G$ has no b-colouring with five colours.}
Otherwise, we use Claim~\ref{cl-P3P1-1} to find that $m_i=p_i+1$ for each $i$ ($1\leq i\leq r$). 
As the vertices in $T_i$ are exactly the vertices of maximum degree in $G_i$, this implies the following:

\begin{claim}\label{cl-P3P1-3}
    Every $G_i$ $(1\leq i\leq r)$ is tight and has $T_i$ as its set of dense vertices with $|T_i|=m_i$. 
\end{claim}

\noindent
We now prove one more claim:

\begin{claim}\label{cl-P3P1-4}
    $G$ has a tight b-colouring if and only if each $G_i$ has a tight b-colouring.
\end{claim}
\begin{claimproof}
    ($\Rightarrow$) Suppose that $G$ has a tight b-colouring $c$. Let $S_i$ be the set of colours used by $c$ in $G_i$. As $T_i$ is the set of dense vertices of $G$ in $G_i$, every vertex of $T_i$ is assigned a unique colour of $S_i$, and $|S_i|=m_i$. Since $G_i$ is complete to $G_j$, it follows that $S_i\cap S_j=\emptyset$, for $1\leq i, j\leq r$ and $i\neq j$. Hence, each vertex in $T_i$ picks up the colours of $S_i$ from its neighbours inside $G_i$. Therefore, the restriction of $c$ to $V(G_i)$ is a tight b-colouring of $G_i$.
    
    \medskip
    \noindent
    ($\Leftarrow$) Suppose that $G_i$ has a tight b-colouring $c_i$ for each $i$ ($1\leq i\leq r$). We will construct a tight b-colouring $c$ of $G$. Now, assume that $c_i$ is using colours $1,\dots, m_i$ in $G_i$, for each $i$ ($1\leq i\leq r$). Next, since $V(G_i)$ is complete to $V(G_j)$, for $1\leq i,j\leq r$ and $i\neq j$, it follows that $G_i$ and $G_j$ must have different colour sets in any b-colouring of $G$. Thus, let $d_i=\sum_{j=1}^i m_j$, for each $i$ ($1\leq i\leq r$), and $d_0=0$. Then, assign colours $c(v)=c_i(v)+d_{i-1}$, where $v\in V_i$, for $1\leq i\leq r$. Note that this assignment of colours allows $G_i$ to have colour set $\{d_{i-1}+1,\dots, d_i\}$, for $1\leq i\leq r$. Now, let $v\in V_i$ be a b-chromatic vertex under $c_i$ in $G_i$, for some $i$ ($1\leq i\leq r$). Since $V(G_i)$ is complete to $V(G_j)$, for $1\leq j\leq r$ and $i\neq j$, it follows that $v$ is picking up all the colours in $\{d_{j-1}+1, \dots, d_j\}$, for $j=1,\dots, r$ and $j\neq i$, and therefore $v$ is b-chromatic under $c$. As $m=\sum_{i=1}^r m_i$, every b-chromatic vertex with respect to $c_i$ in $G_i$ is b-chromatic with respect to $c$ in $G$. We conclude that $c$ is a tight b-colouring of $G$.
\end{claimproof}

\medskip
\noindent
Due to Claims~\ref{cl-P3P1-3} and~\ref{cl-P3P1-4}, we may consider the 
graphs $G_1,\ldots, G_r$ separately. 
So, let $G_i$ be the co-component of $G$ under consideration. 
By Corollary~\ref{c-olariu}, $G_i$ is $3P_1$-free or a disjoint union of complete graphs.
In the first case, we apply Lemma~\ref{l-tight2p2p1}, as every $3P_1$-free graph is $(2P_2+P_1)$-free.
In the second case, we apply Lemma~\ref{l-p4}, as every disjoint union of complete graphs is $P_4$-free.
\end{proof}

\noindent
Now, we present  
three hardness results for {\sc Tight b-Chromatic~Number}. The first one is for line graphs and due to Campos et al.~\cite{CLMSSS15} who proved it for {\sc b-Chromatic Number}.

\begin{lemma}[\cite{CLMSSS15}]\label{l-claw}
{\sc Tight b-Chromatic Number} is \NP-complete for line graphs, and thus for claw-free graphs.
\end{lemma}

\begin{proof}
 The result follows by inspection of the graph constructed in the proof of Theorem~1 of Campos et al.~\cite{CLMSSS15}.  Here, the authors construct a graph $H$ as an instance of {\sc b-Chromatic Index}, which is the counterpart of {\sc b-Chromatic Number} for edge colouring.  If we let $G=L(H)$, it may be verified from the authors' construction and their accompanying arguments that $G$ is tight, and indeed that the target number of colours to be used in the b-colouring is equal to $m(G)$. Hence, $G$ is an instance of {\sc Tight b-Chromatic Number}.
\end{proof}

\noindent
Our next two hardness results are obtained by modifying a reduction from the proof of \NP-hardness of {\sc Tight b-Chromatic Number} in 
graphs that are the disjoint union of a connected split graph and three stars 
given by Havet, Sales and Sampaio~\cite{HSS12}, which we will recall briefly. A {\it $3$-edge colouring} in a graph $G=(V,E)$ is a mapping $c:E\to  \{1,2,3\}$ such that $c(e)\neq c(f)$ for every pair of edges that have a common end-vertex. The {\sc $3$-Edge Colouring} problem is to decide if a given graph has a $3$-edge colouring and is known to be \NP-complete even for cubic graphs~\cite{Ho81}.

The reduction in~\cite{HSS12} is from {\sc $3$-Edge Colouring in cubic graphs}.
First, let $G=(V,E)$ be a cubic graph with vertex set $V=\{u_1,\dots, u_n\}$ and edge set $E=\{e_1,\dots, e_m\}$. Now, define 
\begin{itemize}
\item $V'=\{u_{n+r}:1\leq r\leq 3\}$, and
\item $W=\{w_{n+r}^k:1\leq r\leq 3 \wedge 1\leq k\leq n+2\}$.
\end{itemize}

Let $H$ be a graph with vertex set $V(H)=V\cup E\cup V'\cup W$ and edge set
\begin{align}
    \begin{split}
        E(H) &=\{u_ie_j : 1\leq i\leq n \wedge 1\leq j\leq m \wedge u_i\in e_j\}\\
             &\cup \{u_iu_{i'} : 1\leq i,i'\leq n \wedge i\neq i'\}\\
             &\cup \{u_{n+r}w_{n+r}^k:1\leq r\leq 3 \wedge 1\leq k\leq n+2\}.
    \end{split}
\end{align}

\noindent 
The graph~$H$ is illustrated in Figure \ref{fig:3-edge-3-reg-col-red-chordal}. Notice that $m(H)=n+3$ and $H$ is tight. Furthermore, the graph induced by $V$ and $E$, namely $H[V\cup E]$, is a \emph{split} graph, i.e., a graph whose vertex set can be partitioned into a clique and an independent set. Havet, Sales and Sampaio~\cite{HSS12} proved that $G$ admits a colouring of its edges using three colours if and only if $H$ admits a b-colouring with $n+3$ colours. 

\begin{figure}[t] 
    \centering
    \includegraphics[width=1\linewidth]{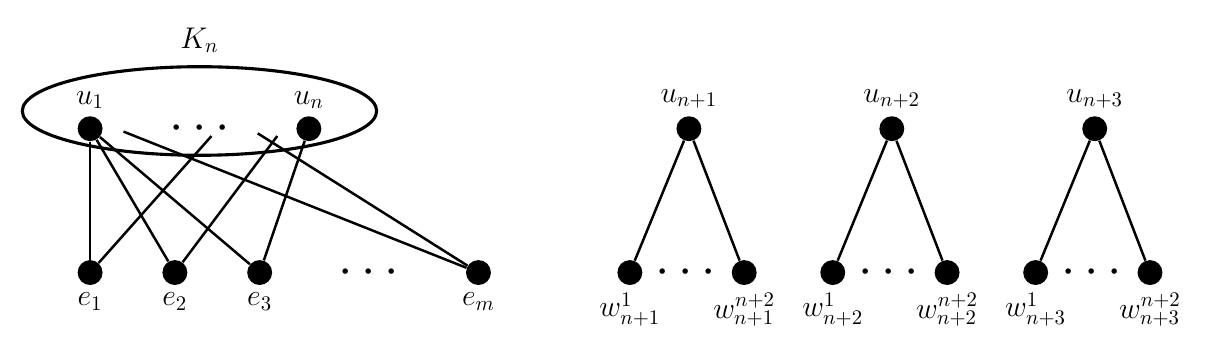}
    \caption{The graph~$H$ from~\cite{HSS12}.}
    \label{fig:3-edge-3-reg-col-red-chordal}
\end{figure}

\begin{lemma}\label{l-3p2}
{\sc Tight b-Chromatic Number} is \NP-complete for $3P_2$-free graphs.
\end{lemma}
\begin{proof}
    We will modify again the graph~$H$ constructed in the reduction given by Havet, Sales and Sampaio~\cite{HSS12} from {\sc $3$-Edge Colouring} in cubic graphs, as discussed above.  
    As in the proof of Lemma \ref{l-claw}, reduce the number of neighbours of $u_{n+r}$ ($1\leq r\leq 3$) from $n+2$ to $n$, i.e., $N(u_{n+r})=\{w_{n+r}^k:1\leq k\leq n\}$. Then, add a clique on the vertices $u_{n+r}$ ($1\leq r\leq 3$) . Let $H'$ be the graph that is obtained from $H$ in this way; $H'$ is illustrated in Figure \ref{fig:3-edge-3-reg-col-red-chordal-mod-3P2-free}.
    \begin{figure}[t]
        \centering
        \includegraphics[width=1\linewidth]{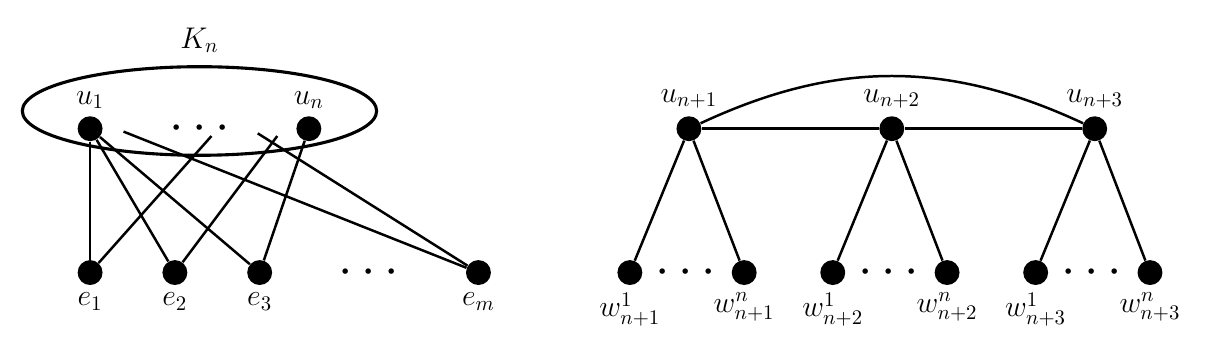}
        \caption{The $3P_2$-free graph~$H'$ used in the proof of Lemma~\ref{l-3p2}.}
        \label{fig:3-edge-3-reg-col-red-chordal-mod-3P2-free}
    \end{figure}

    It is straightforward to observe that $m(H')=n+3$, that $H'$ is tight, and that 
    the original cubic graph $G$ admits a $3$-edge colouring
    if and only if $\varphi(H')=m(H')$ (given that the corresponding property holds for $H$~\cite{HSS12}).
    
    Now, we show that $H'$ is $3P_2$-free.
    Suppose for a contradiction that $H'$ \emph{does} contain an induced $3P_2$, i.e., $f_1,f_2$ and $f_3$ are three edges in $H'$ such that no end-vertex from $f_j$ is adjacent to an end-vertex from $f_{j'}$ $(1\leq j,j'\leq 3)$.    
    
    Let $W'=\{w_{n+r}^k : 1\leq r\leq 3\wedge 1\leq k\leq n\}$. Since
    $E$ and $W'$
    are independent sets, no edge in $H'$ has its two end-vertices both in $E$ or both in $W'$. Therefore, every edge in $H'$ must have at least one of its end-vertices in either $V$ or $V'$. Assume without loss of generality that $f_1$ has an end-vertex in $V$.  If $f_2$ has an end-vertex in $V$ then $f_1$ and $f_2$ have adjacent end-vertex, a contradiction. Hence, $f_2$ has an end-vertex in $V'$.   
    If $f_3$ has an end-vertex in $V$,  
    $f_3$ and $f_1$ have adjacent end-vertex, whereas if $f_3$ has an end-vertex in $V'$, $f_3$ and $f_2$ have adjacent end-vertex.  As both of these are a contradiction, it follows that $H'$ is $3P_2$-free.  \end{proof}

\begin{lemma}\label{l-2p3}
{\sc Tight b-Chromatic Number} is \NP-complete for $2P_3$-free graphs.
\end{lemma}
\begin{proof}
    We will modify again the graph~$H$ constructed in the reduction given by Havet, Sales and Sampaio~\cite{HSS12} from {\sc $3$-Edge Colouring} in cubic graphs, as discussed above.

    Firstly, we let $H'=H[V\cup E]$ initially, i.e., we remove the three stars on $n+3$ vertices from $H$. Now, add three cliques $A,B$ and $C$, where $A=\{a_1,\dots,a_{m+1}\}$, $B=\{b_1,b_2,b_3\}$ and $C=\{c_1,\dots,c_n\}$. Next, add edge $u_ia_r$ for every $u_i\in V$ and $a_r\in A$, edge $a_rb_s$ for every $a_r\in A$ and $b_s\in B$, edge $b_sc_t$ for every $b_s\in B$ and $c_t\in C$, and edge $c_te_j$ for every $c_t\in C$ and $e_j\in E$, to $H'$. The constructed graph~$H'$ is illustrated in Figure \ref{fig:2P3-free-reduction-gadget}. 
    
    \begin{figure}[t]
        \centering
        \includegraphics[width=0.75\linewidth]{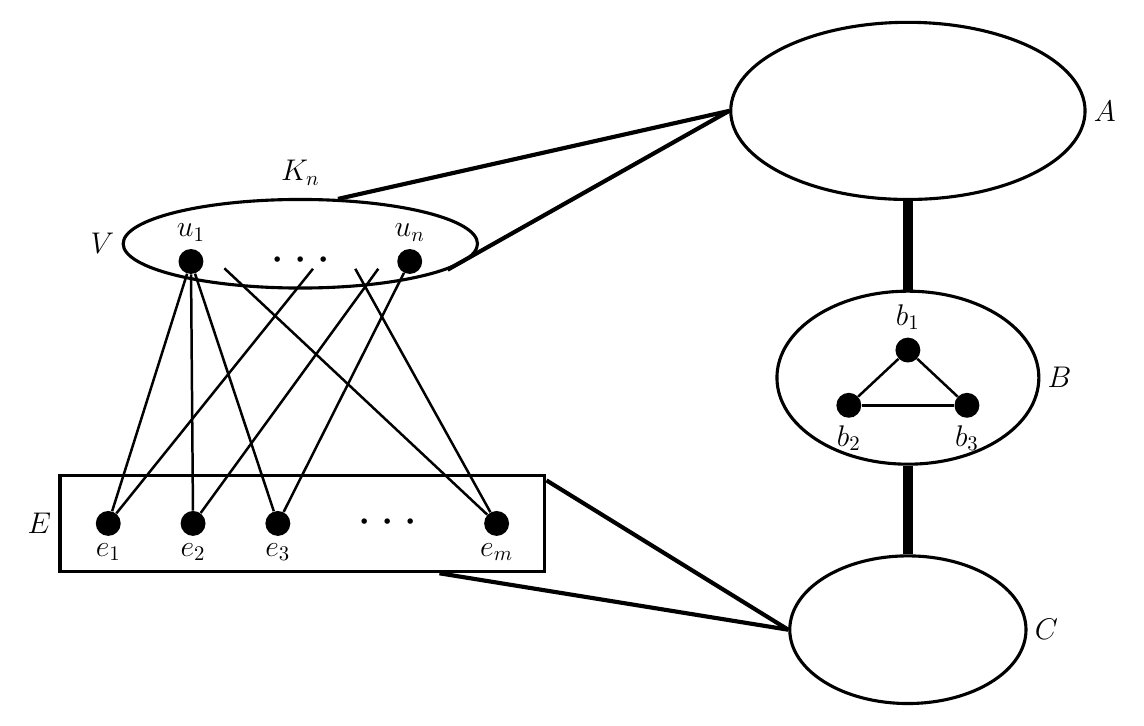}
        \caption{The $2P_3$-free graph~$H'$ used in the proof of Lemma~\ref{l-2p3}.}
        \label{fig:2P3-free-reduction-gadget}
    \end{figure}
    
    Now observe that:
    \begin{itemize}
        \item for $u_i\in V$, $d(u_i)= |V|-1 + 3 + |A| = m+n+3$,
        \item for $a_r\in A$, $d(a_r)=|A|-1 + |V| + |B| = m+n+3$,
        \item for $b_s\in B$, $d(b_s)=|B|-1+|A|+|C|=m+n+3$,
        \item for $c_t\in C$, $d(c_t)=|C|-1 +|B|+|E| = m+n+2$,
        \item for $e_j\in E$, $d(e_j)=2+n<m+n+3$.
    \end{itemize}
    It follows that $m(H')=m+n+4$ since there are $|V\cup A\cup B|=m+n+4$ vertices with degree $m+n+3$; moreover $H'$ is tight. Now, we claim that the original cubic graph $G$ has a $3$-edge colouring if and only if $\varphi(H')=m(H')$.

    \medskip
    \noindent
    ($\Rightarrow$) Suppose that $G$ has an edge colouring $c_G$ using colours $1$, $2$ and $3$. Define a function $c_{H'}$ in $H'$ as follows.  Let $c_{H'}(e_j)=c_G(e_j)$ for all $e_j\in E$, $c_{H'}(u_i)=3+i$ for all $u_i\in V$, $c_{H'}(a_r)=n+3+r$ for all $a_r\in A$, $c_{H'}(b_s)=s$ for all $b_s\in B$, and $c_{H'}(c_t)=3+t$ for all $c_t\in C$.  It is straightforward to verify that $c_{H'}$ is a colouring of $H'$.
     
    It follows that vertices in $B$ are b-chromatic for colours $1$, $2$ and $3$. Next, observe that each vertex $u_i\in V$ is picking up colours $1$, $2$ and $3$ from vertices in $E$ by the properties of $c_G$. Furthermore, $u_i$ is picking up colours $\{4,\dots, n+3\}\setminus \{c_{H'}(u_i)\}$ from vertices in $V\setminus \{u_i\}$ and picking up colours $\{n+4,\dots,m+n+4\}$ from vertices in $A$. It follows that vertices in $V$ are b-chromatic for colours $4,\dots, n+3$. Lastly, observe that each vertex $a_r\in A$ is picking up colours $\{1,2,3\}$ from vertices in $B$, colours $\{4,\dots, n+3\}$ from vertices in $V$ and colours $\{n+4, \dots,m+n+4\}\setminus \{c_{H'}(a_r)\}$ from vertices in $A\setminus \{a_r\}$. It follows that vertices in $A$ are b-chromatic for colours $n+4,\dots,m+n+4$. Hence, $\varphi(H')=m(H')$.
    
    \medskip
    \noindent
    ($\Leftarrow$) Suppose that $\varphi(H')=m(H')$ and let $c_{H'}$ be a b-colouring of $H'$ with $m+n+4$ colours.  The only dense vertices are those in $A\cup B\cup V$, each of which has degree exactly $m+n+3$. We may assume without loss of generality that $c_{H'}(b_s)=s$ for each $b_s\in B$, $c_{H'}(u_i)=3+i$ for each $u_i\in V$ and $c_{H'}(a_r)=n+3+r$ for each $a_r\in A$. 
    As $c_{H'}$ is a b-colouring, every vertex 
    $u_i\in V$ needs to pick up colours $1$, $2$ and $3$, which it can only obtain from its three neighbours in $E$.  Define a mapping $c_G$ in $G$ by letting $c_G(e_j)=c_{H'}(e_j)$ for each $e_j\in E$.  Clearly $c_G$ is then a $3$-edge colouring of $G$.

\medskip
\noindent
    It remains to show that $H'$ is $2P_3$-free. For a contradiction, let
    $F=F_1+F_2$,
    where both $F_1$ and $F_2$ are isomorphic to $P_3$,
    be an induced subgraph of $H'$. 
    As $H'[B\cup C\cup E]$ is a split graph, which is $2P_2$-free, $F$ has at least one vertex $u$ of $A\cup V$,
    say $u\in V(F_1)$.
    As $A\cup V$ is a clique, at least one vertex~$v\in V(F_1)\setminus \{u\}$ belongs to $B\cup C\cup E$, and moreover, $V(F_2)\subseteq B\cup C\cup E$. As every vertex of $C$ dominates $G[B\cup C\cup E]$, the latter implies that $v$ belongs to $B\cup E$.
    First suppose that $v\in B$.
    As $B$ is complete to $C$, this means that $V(F_2)\subseteq E$.
    However, this is not possible, as $E$ is an independent set.
    Hence, $v\in E$. Since $E$ is complete to~$C$, this means that $V(F_2)\cap C=\emptyset$.
    Hence, as $F_2$ is connected, $V(F)\subseteq B$ or $V(F)\subseteq E$. Both cases are not possible, as $B$ is a clique and $E$ is an independent set. We conclude that $H'$ is $2P_3$-free.
\end{proof}

\noindent
We are now ready to show Theorem~\ref{t-tight}, which we restate below.

\ttight*

\begin{proof}
Let $H$ be a graph.
If $H\ssi P_4$,  $H\ssi 2P_2+P_1$,  
or $H\ssi P_3+P_1$,  we apply Lemma~\ref{l-p4},~\ref{l-tight2p2p1}
or~\ref{l-tightp13}, respectively.
We now consider the hard cases.
First suppose $H$ is not a linear forest, so $H$ has an induced cycle or an induced claw.
If $H$ has an induced~$C_r$, then we apply Lemma~\ref{l-bip} if $r=3$ and Lemma~\ref{l-split} if $r\geq 4$.
If $H$ has an induced claw, then we apply Lemma~\ref{l-claw}.
Now suppose $H$ is a linear forest.
If $P_5\ssi H$, $3P_2\ssi H$ or $2P_3\ssi H$, then we apply Lemma~\ref{l-split},~\ref{l-3p2} or~\ref{l-2p3}, respectively.  
\end{proof}

\section{The Proof of Theorem~\ref{t-fall}}\label{s-fall}

Again we need some known results as lemmas, starting with a result of Silva~\cite{Si19} who proved that every chordal graph with a fall colouring is fall-unique and that deciding if the fall spectrum is non-empty is \NP-complete for chordal graphs.

\begin{lemma}[\cite{Si19}]\label{l-chordal}
{\sc Fall Chromatic Number} and {\sc Fall Achromatic Number} are \NP-hard for chordal graphs, that is, $(C_4,C_5,\ldots)$-free graphs.
\end{lemma}

\noindent
We also need another result of Silva~\cite{Si19}.

\begin{lemma}[\cite{Si19}]\label{l-fallp4}
{\sc Fall Colouring}, and thus {\sc Fall Chromatic Number} and {\sc Fall Achromatic Number}, are polynomial-time solvable for $P_4$-sparse graphs, and thus for $P_4$-free graphs.
\end{lemma}

\noindent
The following two hardness results were originally proven for {\sc Fall $3$-Colouring}, but a simple trick ensures that the problem stays \NP-complete for our purposes.

\begin{lemma}[\cite{LL09}]\label{l-fallbip}
{\sc Fall Chromatic Number} and {\sc Fall Achromatic Number} are \NP-hard for $C_3$-free graphs.
\end{lemma}

\begin{figure}[t]
    \centering
    \includegraphics[width=0.5\linewidth]{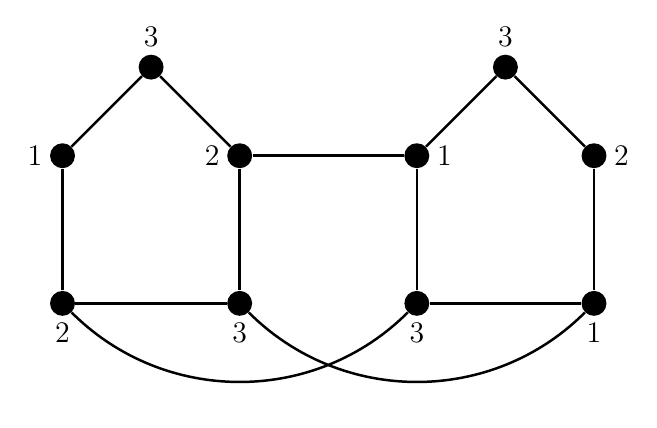}
    \caption{The $C_3$-free gadget $G'$  with $\mathcal{F}(G')=\{3\}$ from the proof of Lemma~\ref{l-fallbip}.}
    \label{fig:C3-free-gadget-fall-spe-3}
\end{figure}

\begin{proof}
Laskar and Lyle~\cite{LL09} proved that {\sc Fall $3$-Colouring} is \NP-complete even for bipartite graphs. Let $G$ be a bipartite graph, and let $G'$ be the $10$-vertex gadget depicted in Figure \ref{fig:C3-free-gadget-fall-spe-3}. Note that $G'$ is $C_3$-free and $\mathcal{F}(G')=\{3\}$.
It remains to observe that $G+G'$ is $C_3$-free, and also that $G$ has a fall $3$-colouring if and only if $G+G'$ is fall-unique.  
\end{proof}

\begin{lemma}[\cite{LM20}]\label{l-line}
{\sc Fall Chromatic Number} and {\sc Fall Achromatic Number} are \NP-hard for line graphs, and thus for claw-free graphs.
\end{lemma}

\begin{proof}
We recall that Lauri and Mitillos~\cite{LM20} proved that {\sc Fall $3$-Colouring} is \NP-complete even for $4$-regular line graphs 
(see the proof of Theorem~15 in~\cite{LM20}). Let $G$ be a line graph. Note that $C_3$ is the line graph of the claw and ${\cal F}(C_3)=\{3\}$. Hence, $G+C_3$ is a line graph, and $G$ has a fall $3$-colouring if and only if $G+G'$ is fall-unique.  
\end{proof}

\noindent
We now prove a new polynomial-time solvability result using Corollary~\ref{c-olariu}. 

\begin{lemma}\label{l-p13}
{\sc Fall Colouring}, and thus {\sc Fall Chromatic Number} and {\sc Fall Achrom\-atic Number}, is polynomial-time solvable for $(P_3+P_1)$-free graphs.
\end{lemma}

\begin{proof}
Let $G$ be a $(P_3+P_1)$-free graph with co-components $G_1, \ldots, G_r$ for some $r\geq 1$.
By Corollary~\ref{c-olariu}, every $G_i$ is $3P_1$-free or a disjoint union of complete graphs. We will show below that $G$ is fall-unique if $G$ has a fall colouring, and that in that case we can compute the singleton in ${\cal F}(G)$ in polynomial time. This immediately proves the lemma.

By definition, $V(G_i)$ is complete to $V(G_j)$ if $i\neq j$. Hence, every fall colouring~$c$ of $G$ (if one exists) colours the vertices of each $G_i$ with different colours, that is, 
$c(V(G_i))\cap c(V(G_j))=\emptyset$ if $i\neq j$. Hence, $G$ has a fall colouring if and only if each $G_i$ has a fall colouring, and we can consider the graphs $G_1,\ldots, G_r$ separately. 

Let $G_i$ be the co-component of $G$ under consideration for some $i$ ($1\leq i\leq r$).
We consider two cases.

\medskip
\noindent
{\bf Case 1.} $G_i$ is $3P_1$-free.\\
We observe that the unique vertex in every singleton colour class of every fall colouring of $G_i$ must be a dominating vertex of $G_i$. Hence, we first determine, in polynomial time, the set $S$ of all dominating vertices of $G_i$. As $G_i'=G_i-S$ is also $3P_1$-free, and $G_i$ has a fall colouring with $k$ colours if and only if $G_i'$ has a fall colouring with
$k-|S|$ colours, we may now safely consider $G'_i$ instead of $G_i$. 

As $G_i'$ has no dominating vertices and is $3P_1$-free, every colour class of every fall colouring of $G_i'$ (if one exists) has size exactly~$2$. Hence, $G_i'$ is fall-unique. Moreover, the colour classes of every fall colouring of $G_i$ form the edges of a perfect matching in $\overline{G_i'}$. 

We claim that it suffices to check if $\overline{G_i'}$ has a perfect matching, which can be done in polynomial time~\cite{Ed65}. To see this, suppose $M$ is a perfect matching of $\overline{G_i'}$, say $M=\{u_1v_1,\ldots, u_pv_p\}$ where $p=\frac{1}{2}|V(G_i')|$. Consider the colouring $c$ of $G_i'$ with colour classes $\{u_h,v_h\}$ for $1\leq h\leq p$. If there exists a vertex $w\in \{u_h,v_h\}$ not adjacent to $\{u_j,v_j\}$ for some $j\neq h$, then $\{u_j,v_j,w\}$ induces a $3P_1$ in $G_i'$, a contradiction. Hence, $c$ is a fall colouring of $G_i$, and $\mathcal{F}(G_i)=\{|S|+p\}$.

\medskip
\noindent
{\bf Case 2.} $G_i$ is a disjoint union of complete graphs.\\
As each connected component $K$ of $G_i$ is complete, the vertices of $K$
must all receive a different colour. Hence, $G_i$ is fall-unique if it has a fall colouring. The latter holds if and only if every connected component of $G_i$ has the same size $p$. Moreover $\mathcal{F}(G_i)=\{p\}$. 
\end{proof}

\noindent
The final lemma we need is a new hardness result based on a construction from~\cite{PPR19} for testing if a graph has a critical vertex. In turn, the construction in~\cite{PPR19} was based on a construction of Kr\'al et al.~\cite{KKTW01} for proving that {\sc Chromatic Number} is \NP-hard for  $(C_5,2P_2,P_2+2P_1,4P_1)$-free graphs.

\begin{figure}[t]
\begin{center}
\includegraphics[keepaspectratio=true, width=8cm]{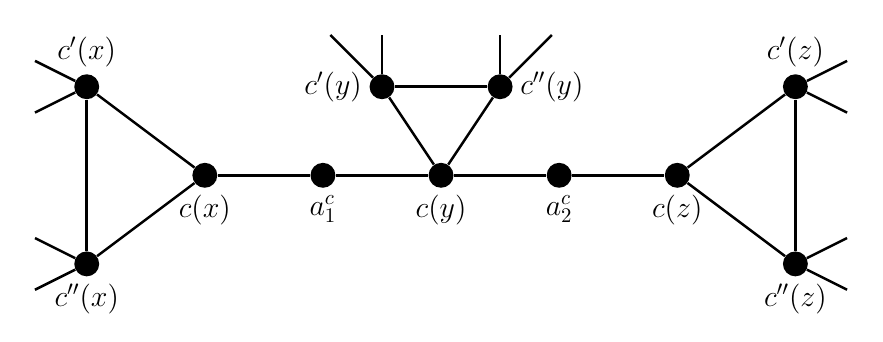}
\caption{The clause gadget $G_c$ and three variable gadgets $Q_x$, $Q_y$ and $Q_z$ (figure taken from~\cite{PPR19}). Note that one of the two edges incident to $c'(x)$ that ``miss'' their other end may not exist; the same holds for the other pairs of such edges.}\label{fig:clause}
\end{center}
\end{figure}

\begin{lemma}\label{l-hardnew}
{\sc Fall Chromatic Number} and {\sc Fall Achromatic Number} are \NP-hard for $(C_5,2P_2,P_2+2P_1,4P_1)$-free graphs.
\end{lemma}

\begin{proof}
We reduce from the following problem.
Let $C$ be a set of clauses, each of which consists of three distinct positive literals from a variable set~$X$. Assume also that each variable of $X$ occurs in exactly three clauses (and consequently $|C|=|X|$). By putting the clauses in conjunctive normal form, we obtain a {\it $(3,3)$-monotone formula $\Phi$}. 
A {\it truth assignment} of $\Phi$ maps each variable of $X$ to either true or false.
We say that $\Phi$ is {\it $1$-satisfiable} if $\Phi$ has a truth assignment that satisfies each clause by exactly one variable, that is, maps exactly one variable of each clause to true.
The problem {\sc Monotone $1$-in-$3$-SAT}, also known as {\sc Positive $1$-in-$3$-SAT},
is to decide whether a given $(3,3)$-monotone formula is $1$-satisfiable.
Moore and Robson~\cite{MR01} proved that this problem is \NP-complete.

From a $(3,3)$-monotone formula $\Phi$ with clause set $C$ of size~$n$ and variable set $X$ of size~$n$, we construct the graph~$G$ from~\cite{PPR19} (see also Figure~\ref{fig:clause}):
\begin{itemize}
\item For each clause $c\in C$, the clause gadget $G_c=(V_c,E_c)$ is the $5$-vertex path $c(x)a_1^cc(y)a_2^cc(z)$, where $c(x),c(y),c(z)$ are {\it $c$-type} vertices corresponding to the three (positive) literals $x,y,z$ of $c$, which we ordered arbitrarily, and $a_1^c, a_2^c$ are {\it $a$-type} vertices. 
\item For each variable $x\in X$, the variable gadget $Q_{x}$ is the triangle on vertices $c(x),c'(x),c''(x)$, where $c,c',c''$ are the three clauses containing~$x$. \end{itemize}
As noted in~\cite{PPR19}, $\vert{V(G)}\vert=5n$ and moreover, $G$ is $(C_5,C_4,\overline{2P_1+P_2},K_4)$-free with $\omega(G)=3$. Hence, $\overline{G}$ is  $(C_5,2P_2,2P_1+P_2,4P_1)$-free.

We claim that $\Phi$ is $1$-satisfiable if and only if ${\cal F}(\overline{G})=\{\frac{7}{3}n\}$, which would prove the lemma. In order to show this statement we will combine some new arguments with Claims~1--3 from the proof of Theorem~5 in~\cite{PPR19}. We give these claims below in order, but we modified the statement of Claim~3 from~\cite{PPR19}, and therefore also give a proof of it.

\begin{claim}\label{cl-fall-1}
    There exists a minimum clique cover of $G$, in which each vertex $a_i^c$ ($c\in C$, $1\leq i\leq 2$) is covered by a clique of size~$2$, implying that $\sigma({G})\ge \frac{7}{3}n$.
\end{claim}

\begin{claim}\label{cl-fall-2}
    $\Phi$ is $1$-satisfiable if and only if $\sigma({G})=\frac{7}{3}n$.
\end{claim}

\begin{claim}\label{cl-fall-3}
    If $G$ has a clique cover ${\cal K}=\{K_1,\ldots,K_{\frac{7}{3}n}\}$, then each $K_i$ with only $c$-type vertices has size~$3$ and each $K_i$ that contains an $a$-type vertex has size~$2$.
\end{claim}
\begin{claimproof}
    We repeat the arguments of~\cite{PPR19}. By Claim~\ref{cl-fall-1}, we have that
    $\sigma({G})\ge\frac{7}{3}n$. As $\vert{\cal K}\vert=\frac{7}{3}n$, this means that ${\cal K}$ is a minimum clique cover.
    For $K_i\in \cal K$ and $v\in K_i$, we define the {\it weight} $w_v={1/\vert K_i\vert}$. Note that, as $\omega({G})= 3$, we have $w_v\in\{\frac{1}{3},\frac{1}{2},1\}$. Moreover, as shown in~\cite{PPR19}:
    \[\displaystyle\sum_{v\in V_c}w_v=\frac{7}{3}\; \mbox{for all}\; c\in C.\] 
    Every $a_i^c$ has exactly two neighbours and these neighbours are not adjacent. Hence, $w_{a_i^c}\in \{\frac{1}{2},1\}$ and every $K_i$ with an $a$-type vertex has size~$1$ or~$2$. However, as shown in~\cite{PPR19}, if some $a_i^c$ has $w_{a_i^c}=1$, then $\sum_{v\in V_c}w_v>\frac{7}{3}$, which contradicts the above. Therefore, each $K_i$ with an $a$-type vertex has size~$2$. 
    
    Now consider a clause $c\in C$. From the above, $w_{a_1^c}=w_{a_2^c}=\frac{1}{2}$. Hence, at least two vertices of $c(x),c(y),c(z)$ have weight $\frac{1}{2}$. As $\sum_{v\in V_c}w_v=\frac{7}{3}$, this means that the remaining third vertex of $c(x),c(y),c(z)$, which does not belong in a clique of ${\cal K}$ together with an $a$-type vertex, must have weight $\frac{1}{3}$. Hence, every $K_i$ with only $c$-type vertices has size~$3$.
\end{claimproof}

\noindent
We also prove a new claim.

\begin{claim}\label{cl-fall-4}
    $\overline{G}$ has no fall colouring with more than $\frac{7}{3}n$ colours.
\end{claim}
\begin{claimproof}
    For a contradiction, assume that $\overline{G}$ has a fall colouring $f$ with more than $\frac{7}{3}n$ colours. Let ${\cal K}_f$ be the corresponding clique cover of $G$ obtained from the colour classes of $f$. As every vertex in~$G$ has at least one neighbour in~$G$, no vertex of $\overline{G}$ is a dominating vertex of $\overline{G}$. 
    In a fall colouring, the single vertex of a singleton colour class must be dominating.
    As $f$ is a fall colouring, we therefore conclude that every colour class of~$f$, and thus every clique in ${\cal K}_f$, has size at least~$2$. 
    
    Now consider a clause $c\in C$. From the above, $a_1^c$ and $a_2^c$ belong, together with a vertex of $c(x),c(y),c(z)$, to two (different) cliques of ${\cal K}_f$ of size~$2$.
    Consider the remaining third vertex from $c(x),c(y),c(z)$, which belongs
    to a clique $K\in {\cal K}_f$ with only $c$-type vertices. If $K$ has size~$2$, then $K$ is properly contained in a triangle~$T$ of $c$-type vertices. However, we now find that in $\overline{G}$, the vertex of $V(T)\setminus K$ is not adjacent to any of the two vertices of $K$, which form a colour class of $f$. This means that $f$ is not a fall colouring of $\overline{G}$, a contradiction. Hence, $K$ must have size~$3$.
    
    From the above, we conclude that every clique of ${\cal K}_f$ with only $c$-type vertices has size~$3$, and that every clique of ${\cal K}_f$ with an $a$-type vertex has size~$2$.
    There are $2n$ $a$-type vertices in total, and each $a$-type vertex belongs to a clique in ${\cal K}_f$ with exacty one $c$-type vertex. Hence, there are $2n$ cliques in ${\cal K}_f$ with an $a$-type vertex. Consequently, there are $3n-2n=n$ remaining $c$-type vertices, which all belong to cliques in ${\cal K}_f$ that consist of three $c$-type vertices. Hence, the number of such cliques in ${\cal K}_f$ is $\frac{1}{3}n$. This means that ${\cal K}$ consists of $2n+\frac{1}{3}n=\frac{7}{3}n$ cliques in total, so $f$ is a fall colouring of $\overline{G}$ with only $\frac{7}{3}n$ colours, a contradiction.
\end{claimproof}

\noindent
We are now ready to prove that $\Phi$ is $1$-satisfiable if and only if ${\cal F}(\overline{G})=\{\frac{7}{3}n\}$.

\medskip
\noindent
First suppose that $\Phi$ is $1$-satisfiable. By Claim~\ref{cl-fall-2},  $\sigma({G})=\frac{7}{3}n$, implying that $G$ has a clique cover ${\cal K}=\{K_1,\ldots,K_{\frac{7}{3}n}\}$. By Claim~\ref{cl-fall-3}, 
each $K_i\in\cal K$ with only $c$-type vertices has size~$3$ and each $K_i$ that contains an $a$-type vertex has size~$2$. This means that no vertex of $V(G)$ is complete to a clique in ${\cal K}$.
Consequently, ${\cal K}$ corresponds to a fall colouring of $\overline{G}$, and thus $\frac{7}{3}n\in {\cal F}(\overline{G})$.

By Claim~\ref{cl-fall-1}, we have that $\chi(\overline{G})=\sigma({G})\ge\frac{7}{3}n$. By Observation~\ref{o-known}, this means that every integer smaller than $\frac{7}{3}n$ does not belong to ${\cal F}(\overline{G})$. By Claim~\ref{cl-fall-4}, the same holds for every integer greater than $\frac{7}{3}n$. Hence, ${\cal F}(\overline{G})=\{\frac{7}{3}n\}$.

\medskip
\noindent
Now suppose that $\Phi$ is not $1$-satisfiable. By Claims~\ref{cl-fall-1} and~\ref{cl-fall-2}, we find that $\chi(\overline{G})=\sigma({G})>\frac{7}{3}n$. By Observation~\ref{o-known}, every fall colouring of 
$\overline{G}$ must use at least $\chi(\overline{G})>\frac{7}{3}n$ colours. By Claim~\ref{cl-fall-4}, this means that ${\cal F}(\overline{G})=\emptyset$.  
\end{proof}

\noindent
We are now ready to show Theorem~\ref{t-fall}, which we restate below.

\tfall*

\begin{proof}
Let $H$ be a graph. If $H$ has an induced $C_3$,  we apply Lemma~\ref{l-fallbip}.
If $H$ has an induced $C_r$ for some $r\geq 4$, we apply Lemma~\ref{l-chordal}.
From now on, assume $H$ is $(C_3,C_4,\ldots)$-free, so $H$ is a forest.
If $H$ has an induced claw, then we apply Lemma~\ref{l-line}. Now also assume $H$ is claw-free. This means that $H$ is a linear forest.

If $H$ has at least four connected components, then $H$ contains an induced~$4P_1$ and we apply Lemma~\ref{l-hardnew}.
Suppose $H$ has exactly three connected components. If $H$ has an edge, then $H$ contains an induced $P_2+2P_1$, and we apply Lemma~\ref{l-hardnew} again. Otherwise $H=3P_1\ssi P_3+P_1$, meaning we can apply Lemma~\ref{l-p13}.

Now suppose $H$ has two connected components. Hence, as $H$ is a linear forest, $H=P_s+P_r$ for some $r,s$ with $r\leq s$. If $r\geq 2$, then $H$ has an induced $2P_2$, and we apply Lemma~\ref{l-hardnew}. Now assume $r=1$. If $s\geq 4$, then $P_2+2P_1\ssi P_4+P_1\ssi H$, and we apply Lemma~\ref{l-hardnew}. If $s\leq 3$, then $H\ssi P_3+P_1$, and we apply Lemma~\ref{l-p13}.

Finally, suppose $H$ is connected. Hence, as $H$ is a linear forest, $H=P_r$ for some $r\geq 1$. If $r\geq 5$, then $H$ has an induced $2P_2$, and we apply Lemma~\ref{l-hardnew}. If $r\leq 4$, then $H\ssi P_4$, and we apply Lemma~\ref{l-fallp4}.  
\end{proof}

\section{Conclusions}\label{s-con}

We extended the known results for {\sc b-Chromatic Number}, {\sc Tight Chromatic Number}, {\sc Fall Chromatic Number} and {\sc Fall Achromatic Number} with several new polynomial-time algorithms and hardness results for $H$-free graphs to obtain three complexity dichotomies and one partial complexity classification. Our results provided the first evidence that the computational complexities of {\sc b-Chromatic Number} and {\sc Tight b-Chromatic Number} may differ in $H$-free graphs, namely if $3P_1\ssi H\ssi 2P_2+P_1$ or $H=P_3+P_1$.
To obtain these new polynomial-time cases, we introduced the technique of tight b-precolouring extension.

As our main open problem, we ask about the computational complexity of {\sc Tight b-Chromatic Number} for $H$-free graphs for the remaining cases not covered by Theorem~\ref{t-tight}: 

\begin{open}\label{o-1}
Determine the complexity of {\sc Tight b-Chromatic Number} for $H$-free graphs if 
\begin{itemize}
\item $H=P_4+P_2+sP_1$ for some $s\geq 0$,
\item $H=P_4+sP_1$ for some $s\geq 1$, 
\item $H=P_3+P_2+sP_1$ for some $s\geq 0$,
\item $H=P_3+sP_1$ for some $s\geq 2$,
\item $H=2P_2+sP_1$ for some $s\geq 2$, 
\item $H=P_2+sP_1$ for some $s\geq 3$,
\item $H=sP_1$ for some $s\geq 4$.
\end{itemize}
\end{open}

\noindent
Due to our systematic approach, we identified the following open problem in particular:
does there exist an integer $s\geq 4$, such that {\sc Tight b-Chromatic Number} is \NP-complete on $sP_1$-free graphs, or equivalently, graphs of independence number at most $s-1$?

Graphs of bounded independence number form a set of well-studied graph classes, for which new complexity results continue to be found (see e.g.~\cite{FGSS,JK26}).
All known hardness gadgets~for {\sc Tight b-Chromatic Number} have arbitrarily large independent sets.
It also does not seem straightforward to adjust the hardness gadget for {\sc b-Chromatic Number} in co-bipartite graphs from~\cite{BSSV15} 
or the hardness gadget 
from Lemma~\ref{l-hardnew}.
Hence, new ideas are needed.

\bibliography{chrom}

\end{document}